\documentclass[submission,copyright,creativecommons]{eptcs}
\providecommand{\event}{LSFA 2021} 
\usepackage[utf8]{inputenc}
\usepackage{amsthm}
\usepackage{amsmath}
\usepackage{amsfonts}
\usepackage{amssymb}
\usepackage{float}
\usepackage[inline]{enumitem}
\usepackage{tasks}
\usepackage{stix}
\usepackage{tikz}
\usepackage{booktabs}
\usepackage{mathtools}
\usepackage{subcaption}

\usetikzlibrary{positioning}
\pgfmathtruncatemacro\distance{1}

\newcommand{\dmneg}{{\sim}}

\newcommand{\ISVar}{\mathbb{IS}}
\newcommand{\PPVar}{\mathbb{PP}}
\newcommand{\PPSix}{\mathbf{PP_6}}
\newcommand{\ISSix}{\mathbf{IS_6}}
\newcommand{\ISlogic}{\mathcal{IS}_\leq}
\newcommand{\ISlogicMC}{\mathcal{IS}_{\leq}^{\mathsf{\sequent}}}
\newcommand{\ISAlg}{\mathbf{IS}}
\newcommand{\PPlogic}{\mathcal{PP}_\leq}
\newcommand{\PPlogicMC}{\mathcal{PP}_{\leq}^{\mathsf{\sequent}}}
\newcommand{\PPOnelogic}{\mathcal{PP}_\top}
\newcommand{\PPOnelogicMC}{\mathcal{PP}_{\top}^{\mathsf{\sequent}}}
\newcommand{\FDELogic}{\mathcal{B}}
\newcommand{\powerset}{\wp}
\newcommand{\sequent}{\rhd}
\newcommand{\notsequent}{\,\smallblacktriangleright}
\newcommand{\SetSet}{\textsc{Set-Set}}
\newcommand{\CalcVar}{\mathsf{R}}
\newcommand{\SetFmla}{\textsc{Set-Fmla}}
\newcommand{\Props}{\mathsf{props}}
\newcommand{\sequentx}[1]{\,\rhd_{#1}\,}
\newcommand{\notsequentx}[1]{\smallblacktriangleright_{#1}}

\newcommand{\cons}{{\circ}}

\newcommand{\inferx}[3][]{\frac{#2}{#3}\,#1}
\newcommand{\Hom}{\mathsf{Hom}}
\newcommand{\DefMat}{\mathfrak{M}}
\newcommand{\End}{\mathsf{End}}

\newcommand{\EqName}[1]{\textbf{#1}}
\newcommand{\Variety}{\mathbb{V}}

\newcommand{\LangAlg}[1]{\mathbf{L}_{#1}(P)}
\newcommand{\LangSet}[1]{L_{#1}(P)}
\newcommand{\DefProp}{p}
\newcommand{\DeffProp}{q}
\newcommand{\Fm}{\varphi}
\newcommand{\Fmm}{\psi}
\newcommand{\DMNeg}{{\sim}}
\newcommand{\PPComp}{\neg}
\newcommand{\DefCon}{\copyright}
\newcommand{\LSig}{\Sigma^{\mathsf{bL}}}
\newcommand{\DMSig}{\Sigma^{\mathsf{DM}}}
\newcommand{\DMoSig}{\Sigma^{\mathsf{PP}}}
\newcommand{\DMnSig}{\Sigma^{\mathsf{IS}}}
\newcommand{\OneAssert}{\vdash^{\top}}

\newcommand{\PPAlg}{\mathbf{PP}}
\newcommand{\DMPPExtend}[1]{#1^{\cons}}
\newcommand{\DMISExtend}[1]{#1^{\nabla}}
\newcommand{\DMReductPP}[1]{\widehat{\DMPPExtend{#1}}}
\newcommand{\DMReductIS}[1]{\widehat{\DMISExtend{#1}}}
\newcommand{\Reduce}[1]{{#1}^{\ast}}

\newcommand{\SuperBelnapCon}{\vdash}
\newcommand{\ReducedMatrices}{\mathsf{Mat^\ast}}
\newcommand{\CL}{\mathcal{CL}}

\DeclareMathOperator{\ConSet}{\mathsf{Cng}}
\DeclareMathOperator{\Logic}{\mathsf{Log}}

\newcommand{\neither}{\ensuremath{\mathbf{n}}}
\newcommand{\both}{\ensuremath{\mathbf{b}}}
\newcommand{\fvalue}{\ensuremath{\mathbf{f}}}
\newcommand{\tvalue}{\ensuremath{\mathbf{t}}}
\newcommand{\efvalue}{\ensuremath{\hat{\mathbf{f}}}}
\newcommand{\etvalue}{\ensuremath{\hat{\mathbf{t}}}}
\newcommand{\aefvalue}{\ensuremath{\hat{\mathbf{f}}}}
\newcommand{\aetvalue}{\ensuremath{\hat{\mathbf{t}}}}
\newcommand{\FourSet}{\mathcal{V}_4}
\newcommand{\SixSet}{\mathcal{V}_6}
\newcommand{\SymbDef}{\,{\coloneqq}\,}
\newcommand{\FDEAlg}{\mathbf{DM}_4}
\newcommand{\KleeneAlg}{\mathbf{K}_3}
\newcommand{\BoolAlg}{\mathbf{B}_2}
\newcommand{\Upset}[1]{{{\uparrow}#1}}
\newcommand{\SymLogEquiv}{{\lhd\rhd}}
\newcommand{\CompSet}{\bigwhitestar}

\newcommand{\LeibCong}[1]{\Omega^{#1}}

\newcommand{\NodeSet}{\mathsf{nds}}
\newcommand{\TreeRel}{\mathsf{\leq}}
\newcommand{\DefTree}{t}
\newcommand{\DefNode}{n}
\newcommand{\Label}[1]{l^{#1}}
\newcommand{\Disc}{\ast}
\newcommand{\Root}[1]{\mathsf{rt}(#1)}
\newcommand{\Children}[1]{\mathsf{chdr}}

\newtheorem{theorem}{Theorem}[section]
\newtheorem{lemma}[theorem]{Lemma}
\newtheorem{definition}[theorem]{Definition}
\newtheorem{corollary}[theorem]{Corollary}
\newtheorem{proposition}[theorem]{Proposition}

\newtheorem{example}[theorem]{Example}

\newcommand{\FmSetA}{\Phi}
\newcommand{\FmSetB}{\Psi}
\newcommand{\FmSetC}{\Pi}
\newcommand{\FmSetD}{\Theta}
\newcommand{\FmSetAnalytic}{\Xi}

\newcommand{\FmA}{\varphi}
\newcommand{\FmB}{\psi}
\newcommand{\FmC}{\phi}

\newcommand\blfootnote[1]{%
  \begingroup
  \renewcommand\thefootnote{}\footnote{#1}%
  \addtocounter{footnote}{-1}%
  \endgroup
}

\title{On Logics of Perfect Paradefinite Algebras}
\author{Joel Gomes \qquad Vitor Greati
\institute{ Programa de Pós-graduação em Sistemas e Computação (PPgSC)\\
  Departamento de Informática e Matemática Aplicada (DIMAp)\\
  Universidade Federal do Rio Grande do Norte (UFRN)\\
Natal -- RN, Brazil}
\email{\{joel.gomes.108,vitor.greati.017\}@ufrn.edu.br}
\and
Sérgio Marcelino
\institute{SQIG -- Instituto de Telecomunicações\\
Dep. de Matemática -- Instituto Superior Técnico\\
Universidade de Lisboa, Portugal}
\email{smarcel@math.tecnico.ulisboa.pt}
\and
João Marcos \qquad Umberto Rivieccio
\institute{Departamento de Informática e Matemática Aplicada (DIMAp)\\
  Universidade Federal do Rio Grande do Norte (UFRN)\\
Natal -- RN, Brazil}
\email{\{jmarcos,urivieccio\}@dimap.ufrn.br}
}
\def\titlerunning{On Logics of Perfect Paradefinite Algebras}
\def\authorrunning{Joel Gomes, Vitor Greati, Sérgio Marcelino, \qquad\\
 João Marcos \& Umberto Rivieccio}
\begin{document}
\maketitle

\begin{abstract}
The present study shows how to enrich De Morgan algebras 
with a \emph{perfection operator} that allows one to express the Boolean properties of negation-consistency and negation-de\-ter\-minedness.
The variety of \emph{perfect paradefinite algebras} thus obtained (\emph{PP-algebras}) is shown to be term-equiv\-a\-lent to the variety of involutive Stone algebras, introduced by R.\ Cignoli and M.\ Sagastume, and more recently studied from a logical perspective by M.\ Figallo-L.\ Cantú and by S.~Marcelino-U.~Rivieccio. 
This equivalence  plays an important role in the investigation of the $1$-assertional logic and of the order-preserving logic associated to  PP-algebras. 
The latter logic (here called~$\PPlogic$) is characterized by a single 6-valued matrix and is shown to be 
a Logic of Formal Inconsistency and Formal Undeterminedness.
We axiomatize 
$\PPlogic$ 
by means of an analytic finite Hilbert-style calculus, and we present an 
axiomatization procedure 
that covers the logics corresponding to other classes of De Morgan algebras 
enriched by a perfection operator.
\end{abstract}

\section{Introduction}
\blfootnote{Vitor Greati acknowledges financial support from
the Coordenação de Aperfeiçoamento de Pessoal de Nível Superior --- Brasil (CAPES) --- Finance Code 001.
João Marcos acknowledges partial
support by Conselho Nacional de Desenvolvimento Científico e Tecnológico (CNPq). Sérgio Marcelino's research was done under the scope of Project UIDB/50008/2020
of Instituto de Telecomunica\c{c}\~oes (IT), financed by the applicable framework
(FCT/MEC through national funds and cofunded by FEDER-PT2020).}
The variety of De Morgan algebras consists of all bounded distributive lattices equipped with a primitive 
involutive negation operation $\DMNeg$ satisfying the
well-known De Morgan laws. 
Such negation need not be Boolean, that is, it may fail to satisy the equations $x\lor\DMNeg x\approx\top$ and $x\land\DMNeg x\approx\bot$, respectively expressing the classical `negation-determinedness' and the `negation-consistency' assumptions.
Involutive Stone algebras (henceforth referred to as IS-algebras) are obtained by enriching De Morgan algebras with a further unary operation~$\nabla$ that allows for the definition of a pseudo-complement operator~$\PPComp$ satisfying the Stone equation $\PPComp x \lor \PPComp\PPComp x \approx \top$ (as well as its dual, $\PPComp x \land \PPComp\PPComp x \approx \bot$).


While the order-preserving logic canonically induced by De Morgan algebras, namely Dunn-Belnap's 4-valued logic \cite{belnap1977},
has been extensively studied over the last four decades, the logic similarly induced by IS-algebras (which we call $\ISlogic$) has only recently attracted due attention \cite{cantu:2019,cantu:2020,marcelino2021}.
The latter studies make (but do not pursue to any significant length) an observation that we shall take as the starting point of our present work, namely that, by replacing $\nabla$ with a unary `consistency operator' (here denoted by~$\circ$), it is possible to view $\ISlogic$ as a Logic of Formal Inconsistency (\cite{03-CCM-lfi}). 


From the point of view of non-classical logics,
some of the most prominent features of $\ISlogic$ are the facts that it is paradefinite \cite{avron2018} (it is, indeed, at once $\DMNeg$-paraconsistent and $\DMNeg$-paracomplete, both properties being inherited from the Dunn-Belnap logic), and yet, with the help of the single connective~$\circ$, it may be seen to be expressive enough so as to fully recover the `lost perfection' of classical negation, 
by being at once $\DMNeg$-gently explosive and $\DMNeg$-gently implosive \cite{jm2005thesis}; 
in other words, $\ISlogic$ may be seen as a Logic of Formal Inconsistency (\textbf{LFI}) and a Logic of Formal Undeterminedness (\textbf{LFU}),
in the sense of~\cite{jmarcos2005}. 
These features, however, are somehow concealed by the usual presentation of IS-algebras in terms of~$\nabla$, an algebraic operator whose significance and philosophical motivations are unclear.%
\footnote{For the 3-valued case, such a `possibility' operator is known at least since \cite{JanLuke1930}, where J.\ \L ukasiewicz notes it has been first defined during one of his 1921 seminars by a student called Tarski.  The lack of a robust \textit{modal} reading for such an operator, however, has caused it to have largely fallen by the wayside over the following decades.}


Logics that allow for the internalization of the very notions of negation-consistency and negation-determinedness at the object-language level have been extensively studied in the last two decades (cf.\ \cite{carnielli2019}, for example, for the so-called `classicality', `restoration', `recapture', or `recovery' operators).
In order to establish a fruitful dialogue with the logical study of negation, we  thus propose 
an alternative rendition of IS-algebras in terms of structures that we christen `perfect paradefinite algebras' (or more briefly PP-algebras), obtained by replacing~$\nabla$ with a primitive perfection operation~$\cons$.
The significance of such an approach, we believe, is not only nor primarily technical; instead, it lies mainly in a clarification of the intuitive meanings associated to the propositional connectives employed to present $\ISlogic$. As explained below, another important consequence of our work will be the possibility of singling out new meaningful logical axioms (expressed in the alternative language we propose for $\ISlogic$)  that present more general logics than (i.e.~weakenings of)~$\ISlogic$. 

The equational characterization we present for PP-algebras will not only guarantee that the corresponding variety is term-equivalent to the variety of IS-algebras
but also highlight the expressive paradefinite character of the order-preserving logic thereby induced ($\PPlogic$). The latter logic will be shown, 
more specifically, 
to constitute a fully self-extensional and non-protoalgebraic member of the families of logics known as \textbf{C}-systems and \textbf{D}-systems (detailed explanations and discussions about the latter classes may be found in \cite{jm2005thesis}).
A~procedure for constructing a PP-algebra using a De Morgan algebra as material is introduced and the logic $\PPlogic$ is shown to be characterizable, like $\ISlogic$, by a single six-element logical matrix. 
Lastly, we  also  provide a well-behaved symmetrical Hilbert-style calculus for the \SetSet{} logics determined by logical matrices based on De Morgan algebras enriched with~$\cons$, as well as conventional Hilbert-style calculi for the \SetFmla{} logics determined by logical matrices with prime filters based on De Morgan algebras enriched with~$\cons$ (and, in particular, an analytical proof system for the logic $\PPlogic$ itself).

    In distinction to what has been done
    in the study of logics associated to IS-algebras~\cite{cantu:2019,cantu:2020,marcelino2021},
    in the present work we take the more general path of first obtaining results on the \SetSet{} order-preserving logic (denoted by $\PPlogicMC{}$)
    and on the 1-assertional logic (denoted by $\PPOnelogicMC{}$) associated to PP-algebras,
    and then specializing them to the corresponding \SetFmla{} logics.
    From a proof-theoretical viewpoint, the present paper may thus also be viewed as providing another illustration (additional to the one in~\cite{marcelino2021}) of the wide range of applicability of the machinery of \SetSet{} Hilbert-style calculi, which has recently been further developed 
    in~\cite{marcelinowollic, marcelino:2019}.
  Indeed,
  having established that
 $\PPlogicMC{}$
  is characterizable by a matrix that is finite and sufficiently expressive,
  the problem of obtaining a finite and analytical  \SetSet{} calculus for it
    can be solved by an application of the algorithm of \cite{marcelino:2019},
   which we used via the implementation of~\cite{vitor2021}.
   In order to make the present paper self-contained, we have, though, also included here the proofs of completeness and analyticity of the \SetSet{} calculus thus obtained.
   The conventional \SetFmla{} (non-analytic) axiomatization of $\PPlogic$ is then obtained from the \SetSet{} axiomatization following the general procedure laid out in~\cite[Th.5.37]{shoesmith_smiley:1978}.

   
   
    %
    

The remainder of the paper is organized as follows.
Section \ref{sec:preliminares} provides the basic notions on algebras and logics, the latter considered from the semantical perspective of logical matrices and from the proof-theoretical
viewpoint through Hilbert-style deductive systems.
Section \ref{sec:paradefinite-stone-algebras} introduces the variety of perfect paradefinite algebras and discusses the semantical aspects of the associated logics. More specifically, we prove that these algebras are term-equivalent to the involutive Stone algebras (Theorem~\ref{the:equiv}),
also enjoying, thus, the property of being generated by a single, six-element algebra (Proposition~\ref{prop:sixgenerated}).
We then show that the order-preserving
and 1-assertional logics associated
to this variety are respectively generated
by a 6-valued and a 3-valued logical matrix (Theorem~\ref{the:order-preserving-six-generated}
and Proposition~\ref{the:one-assert-three-generated}). 
Subsequently, we provide a recipe to endow a De Morgan algebra with a perfection operator (Definition~\ref{def:expcons})
and study the lattice of extensions
of the corresponding \SetFmla{} order-preserving logic, showing that it has at least the cardinality of the continuum (Corollary~\ref{coro:embdsupbelnap}).
We close the section by studying 
paradefinite extensions of Belnap-Dunn's 4-valued logic, showing how to recover classical reasoning via assumptions expressed with the help of the perfection operator (Theorem~\ref{the:datgeneral}).
Section \ref{sec:axiomatization}
provides analytic \SetSet{} Hilbert-style deductive systems for logical matrices
based on De Morgan algebras endowed with a consistency operator (Theorem~\ref{the:expcalculus}),
as well as \SetFmla{} Hilbert-style systems
for logics determined by matrices based on De Morgan algebras and prime filters (Theorem~\ref{the:set-fmla-calculi}).
Finally, Section~\ref{sec:finalremarks}
contains some concluding remarks
and outlines future directions of research.%

\section{Algebraic and logical preliminaries} \label{sec:preliminares}

A \emph{propositional signature}
is a family
$\Sigma \SymbDef \{\Sigma_k\}_{k \in \omega}$,
where each $\Sigma_k$ is
a collection of $k$-ary \emph{connectives}. 
A \emph{$\Sigma$-algebra} is a structure
$\mathbf{A} \SymbDef \langle A, \cdot^\mathbf{A} \rangle$, where $A$ is a non-empty
set called the \emph{carrier} of $\mathbf{A}$ and, for each $\DefCon \in \Sigma_k$, $\DefCon^\mathbf{A} : A^k \to A$ is the \emph{interpretation} of $\DefCon$
in $\mathbf{A}$.
Given a denumerable set $P \supseteq \{p, q, r, x, y\}$, 
the absolutely free algebra over
$\Sigma$ freely generated by $P$,
or simply the
\emph{language} over $\Sigma$ (generated by $P$),
is denoted by $\LangAlg{\Sigma}$,
and its members are called \emph{$\Sigma$-formulas}.
The collection of all propositional variables
occurring in a formula $\Fm \in \LangSet{\Sigma}$
is denoted by $\Props(\Fm)$, and
we let $\Props(\FmSetA) \SymbDef \bigcup_{\Fm\in\FmSetA}\Props(\Fm)$, 
for all $\FmSetA\subseteq\LangSet{\Sigma}$.
Given $\Sigma^\prime \subseteq \Sigma$
(that is, $\Sigma^\prime_k \subseteq \Sigma_k$ for all $k \in \omega$),
the \emph{$\Sigma'$-reduct} of a $\Sigma$-algebra $\mathbf{A}$ is the $\Sigma'$-algebra over
the same carrier of $\mathbf{A}$ that agrees with $\mathbf{A}$ on the interpretation of the connectives in $\Sigma'$.
The collection of homomorphisms between
two $\Sigma$-algebras $\mathbf{A}$
and $\mathbf{B}$ is denoted by $\Hom(\mathbf{A},\mathbf{B})$,
and the collection of mappings that are
structure-preserving 
over $\Sigma' \subseteq \Sigma$ is
denoted by $\Hom_{\Sigma'}(\mathbf{A},\mathbf{B})$.
Furthermore, the set
of endomorphisms on $\mathbf{A}$ is denoted by
$\End(\mathbf{A})$
and each one of the members $\sigma \in \End(\LangAlg{\Sigma})$
is called a \emph{substitution}.
The elements of $\Hom(\LangAlg{\Sigma}, \mathbf{A})$ will sometimes be referred to as \emph{valuations on~$\mathbf{A}$}.
{Given $h,h' \in \Hom(\LangAlg{\Sigma}, \mathbf{A})$,
we shall say that \emph{$h'$ agrees with~$h$ on $\FmSetA \subseteq \LangSet{\Sigma}$} provided that $h'(\FmA)=h(\FmA)$
for all $\FmA \in \FmSetA$.}
In case $p_1,\ldots,p_n$
are the only propositional variables ocurring in $\FmA \in \LangSet{\Sigma}$,
we say that $\FmA$ is $n$-ary and denote by $\FmA^\mathbf{A}$
the $n$-ary operation on $A$
such that,
for all $a_1,\ldots,a_n \in A$,
$\FmA^\mathbf{A}(a_1,\ldots,a_n) = h(\FmA)$, for
an $h \in \Hom(\LangAlg{\Sigma}, \mathbf{A})$ with
$h(p_i) = a_i$ for each
$1 \leq i \leq n$.
Also, if $\FmB_1,\ldots,\FmB_n \in \LangSet{\Sigma}$,
we let $\FmA(\FmB_1,\ldots,\FmB_n)$ denote the formula
$\FmA^{\LangAlg{\Sigma}}(\FmB_1,\ldots,\FmB_n)$.
A~\emph{$\Sigma$-equation} is
a pair $(\FmA,\FmB)$ of $\Sigma$-formulas that we will denote by
$\FmA \approx \FmB$,
and a $\Sigma$-algebra $\mathbf{A}$ is said to \emph{satisfy}
$\FmA \approx \FmB$ if
$h(\FmA)=h(\FmB)$
for every $h \in \Hom(\LangAlg{\Sigma}, \mathbf{A})$. 
We call \emph{$\Sigma$-variety}
the class of all $\Sigma$-algebras that satisfy the same given collection of $\Sigma$-equations; an equation is said to be \emph{valid} in a given variety if it is satisfied by each algebra in this variety. 
The variety generated by
a class $\mathsf{K}$ of $\Sigma$-algebras,
denoted by $\mathbb{V}(\mathsf{K})$,
is the closure of $\mathsf{K}$ under 
homomorphic images, subalgebras 
and direct products.
We write $\ConSet\mathbf{A}$ 
to refer to the collection of all congruence relations
on~$\mathbf{A}$, which is known to form
a complete lattice under inclusion.

In what follows, we assume the reader is familiar with basic notations and terminology of lattice theory~\cite{priestley2002}. We denote by $\LSig$ the signature
containing but two binary connectives,
$\land$ and $\lor$, and
two nullary connectives $\top$ and $\bot$, and by $\DMSig$ the extension of the latter signature by the addition of a unary connective~$\DMNeg$.
Moreover, we let $\DMnSig$ and $\DMoSig$ be the signatures obtained
from $\DMSig$ by adding unary connectives $\nabla$ and $\cons$, respectively.
We provide below
the definitions and some examples of
De Morgan algebras and of involutive
Stone algebras.


\vspace{.5em}
\begin{definition}

Given a $\DMSig$-algebra whose
$\LSig$-reduct is a bounded distributive lattice, we say that it
constitutes a \emph{De Morgan algebra} if it
satisfies the equations:
\begin{table}[H]
\begin{tabular}{@{}ll@{}}
    \textbf{(DM1)} $\DMNeg \DMNeg x \approx x$
    \hspace{1cm}
    &
    \textbf{(DM2)} $\DMNeg (x \land y) \approx \DMNeg x \lor \DMNeg y$\\
\end{tabular}
\end{table}
\end{definition}

\begin{example}
Let $\FourSet \SymbDef \{\tvalue, \both, \neither, \fvalue\}$
and
let $\FDEAlg \SymbDef \langle \FourSet, \cdot^{\FDEAlg} \rangle$
be the $\DMSig$-algebra known as
the Dunn-Belnap lattice, whose interpretations
for the lattice connectives
are those induced by the Hasse diagram in
Figure {\rm \ref{fig:de_morgan}},
and the interpretation for 
$\DMNeg$ is such that
$\DMNeg^{\FDEAlg}\fvalue \SymbDef \tvalue$,
$\DMNeg^{\FDEAlg}\tvalue \SymbDef \fvalue$
and
$\DMNeg^{\FDEAlg} a \SymbDef a$,
for $a \in \{\neither,\both\}$;
as expected, for the nullary connectives, we have
$\top^{\FDEAlg} \;\SymbDef\; \tvalue$
and $\bot^{\FDEAlg} \;\SymbDef\; \fvalue$. In Figure~{\rm\ref{fig:de_morgan}},
besides depicting the lattice structure of $\FDEAlg$, we also show
its subalgebras $\KleeneAlg$ and $\BoolAlg$, which 
coincide with the three-element
Kleene algebra and the two-element Boolean algebra.
These three algebras are the only subdirectly irreducible De Morgan algebras~{\rm\cite{dwinger1975}}.
\end{example}
\begin{figure}[ht]
    \centering
    \begin{subfigure}[b]{.45\textwidth}
    \centering
    \begin{tikzpicture}[node distance=1.25cm]
    \node (0DM4)                      {$\fvalue$};
    \node (BDM4)  [above right of=0DM4]  {$\both$};
    \node (ADM4)  [above left of=0DM4]   {$\neither$};
    \node (1DM4)  [above right of=ADM4]  {$\tvalue$};
    \node (AK3)  [right of=BDM4] {$\neither$};
    \node (0K3)  [below=0.4\distance of AK3]  {$\fvalue$};
    \node (1K3)  [above=0.4\distance of AK3] {$\tvalue$};
    \node (0B2)  [right of=0K3]  {$\fvalue$};
    \node (1B2)  [right of=AK3] {$\tvalue$};
    \node (DM4)  [below=1.5\distance of 0DM4] {$\mathbf{DM}_4$};
    \node (K3)  [below=1.5\distance of 0K3] {$\mathbf{K}_3$};
    \node (B2)  [below=1.5\distance of 0B2] {$\mathbf{B}_2$};
    \draw (0DM4)   -- (ADM4);
    \draw (0DM4)   -- (BDM4);
    \draw (BDM4)   -- (1DM4);
    \draw (ADM4)  -- (1DM4);
    \draw (AK3)  -- (0K3);
    \draw (AK3)  -- (1K3);
    \draw (0B2)   -- (1B2);
    \end{tikzpicture}
    \caption{The subdirectly irreducible De Morgan algebras.}
    \label{fig:de_morgan}
    \end{subfigure}
    \hfill
    \begin{subfigure}[b]{.5\textwidth}
    \centering
    \begin{tikzpicture}[node distance=1.25cm]
    \node (0PP6)                      {$\fvalue$};
    \node (BPP6)  [above right of=0PP6]  {$\both$};
    \node (APP6)  [above left of=0PP6]   {$\neither$};
    \node (1PP6)  [above right of=APP6]  {$\tvalue$};
    \node (00PP6)  [below=0.4\distance of 0PP6] {$\efvalue$};
    \node (11PP6)  [above=0.4\distance of 1PP6] {$\etvalue$};
    \node (APP5)  [right of=BPP6] {$\neither$};
    \node (0PP5)  [below=0.4\distance of APP5]  {$\fvalue$};
    \node (1PP5)  [above=0.4\distance of APP5] {$\tvalue$};
    \node (00PP5)  [below=0.4\distance of 0PP5] {$\efvalue$};
    \node (11PP5)  [above=0.4\distance of 1PP5] {$\etvalue$};
    \node (0PP4)  [right of=0PP5]  {$\fvalue$};
    \node (1PP4)  [right of=APP5] {$\tvalue$};
    \node (00PP4)  [below=0.4\distance of 0PP4] {$\efvalue$};
    \node (11PP4)  [above=0.4\distance of 1PP4] {$\etvalue$};
    \node (00PP3)  [right of=00PP4] {$\efvalue$};
    \node (0PP3) [above=0.4\distance of 00PP3] {$\neither$};
    \node (1PP3) [above=0.4\distance of 0PP3] {$\etvalue$};
    \node (PP6)  [below=0.4\distance of 00PP6] {$\mathbf{IS}_6$};
    \node (PP5)  [below=0.4\distance of 00PP5] {$\mathbf{IS}_5$};
    \node (PP4)  [below=0.4\distance of 00PP4] {$\mathbf{IS}_4$};
    \node (PP3)  [below=0.4\distance of 00PP3] {$\mathbf{IS}_3$};
    \node (00PP2)  [right of=00PP3]  {$\efvalue$};
    \node (11PP2)  [above=0.4\distance of 00PP2] {$\etvalue$};
    \node (PP2)  [below=0.4\distance of 00PP2] {$\mathbf{IS}_2$};
    \draw (0PP6)   -- (APP6);
    \draw (0PP6)   -- (BPP6);
    \draw (BPP6)   -- (1PP6);
    \draw (APP6)  -- (1PP6);
    \draw (0PP6)   -- (00PP6);
    \draw (1PP6)   -- (11PP6);
    \draw (APP5)  -- (0PP5);
    \draw (APP5)  -- (1PP5);
    \draw (0PP5)   -- (00PP5);
    \draw (1PP5)   -- (11PP5);
    \draw (0PP4)   -- (1PP4);
    \draw (0PP4)   -- (00PP4);
    \draw (1PP4)   -- (11PP4);
    \draw (0PP3) -- (00PP3);
    \draw (1PP3) -- (0PP3);
    \draw (00PP2) -- (11PP2);
    \end{tikzpicture}
    \caption{The subdirectly irreducible IS-algebras.}
    \label{fig:pp_algebra}
    \end{subfigure}
    \vspace{-.75em}
    \caption{}
\end{figure}
\vspace{-.75em}



\begin{definition}
Given a $\DMnSig$-algebra whose
$\DMSig$-reduct is a De Morgan
algebra, we say that it
constitutes an \emph{involutive Stone algebra} (\emph{IS-algebra}) if it
satisfies the equations:
\begin{table}[H]
\setlength{\tabcolsep}{3.5pt}
\begin{tabular}{@{}llll@{}}
    \textbf{(IS1)} $\nabla \bot \approx \bot$
    \hspace{5mm}
    &
    \textbf{(IS2)} $x \land \nabla x \approx x$
    \hspace{5mm}
    & 
    \textbf{(IS3)} $\nabla(x \land y) \approx \nabla x \land \nabla y$
    \hspace{5mm}
    &
    \textbf{(IS4)} $\DMNeg\nabla x \land \nabla x \approx \bot$
\end{tabular}
\end{table}
\end{definition}

\begin{example}
\label{ex:issix}
Let $\SixSet \SymbDef \FourSet \cup \{\efvalue,\etvalue\}$
and let $\ISAlg_6 \SymbDef \langle \SixSet, \cdot^{\ISAlg_6} \rangle$
be the $\DMnSig$-algebra
whose lattice structure is depicted in Figure~{\rm \ref{fig:pp_algebra}}
and interprets $\DMNeg$ and $\nabla$ as per the following:
\[
\DMNeg^{\ISAlg_6} a \SymbDef
\begin{cases}
    \DMNeg^{\FDEAlg} a & a \in \FourSet\\
    \efvalue & a = \etvalue\\
    \etvalue & a = \efvalue\\
\end{cases}\qquad
\nabla^{\ISAlg_6} a \SymbDef
\begin{cases}
    \etvalue & a \in \SixSet\setminus\{\efvalue\}\\
    \efvalue & a = \efvalue\\
\end{cases}
\]
\noindent 
The subalgebras of $\ISAlg_6$ exhibited in Figure~{\rm\ref{fig:pp_algebra}} constitute the only subdirectly irreducible IS-algebras~{\rm\cite{cignoli1983}}.
\end{example}

We denote by $\ISVar$ the
variety of IS-algebras.
The following result lists some
equations
satisfied by IS-algebras,
which will be useful for
proving the results
in the next section.
\vspace{.5em}
\begin{lemma}
\label{lem:isident}
The following equations are satisfied by IS-algebras:
\begin{tasks}[style=enumerate](3)
    \task $x \lor \nabla \DMNeg x \approx \top$
    \task $x \land \DMNeg\nabla x \approx \bot$
    \task $\DMNeg\nabla(x \land \DMNeg x) \land \DMNeg x \approx \DMNeg\nabla x$
    \task $\nabla\nabla x \approx \nabla x$
    \task $\nabla\DMNeg\nabla x \approx \DMNeg\nabla x$
    \task $\DMNeg\nabla\DMNeg (x \land y) \approx \DMNeg\nabla\DMNeg x \land \DMNeg\nabla\DMNeg y$
\end{tasks}
\end{lemma}
\begin{proof}
Equation 3 may be proved by using the usual De Morgan algebra equations together with $\nabla x \lor x \approx \nabla x$, an equation that is
easily derivable from \EqName{(IS2)}.
All other equations
follow from Lemma 3.2 in \cite{cantu:2020}.
\end{proof}


\vspace{-.5em}

Here, a \emph{\SetFmla{} logic (over $\Sigma$)}
is a consequence relation
$\vdash$ on $\LangSet{\Sigma}$
and a \emph{\SetSet{} logic (over $\Sigma$)}
is a generalized consequence relation 
$\sequent$ on $\LangSet{\Sigma}$ \cite{humberstone:connectives}.
The \emph{\SetFmla{} companion}
of a \SetSet{} logic $\sequent$
is the \SetFmla{} logic $\vdash_{\sequent}$ 
such that $\FmSetA \vdash_{\sequent} \FmB$
if, and only if, $\FmSetA \sequent \{\FmB\}$.
We will write
$\FmSetA \; \SymLogEquiv \; \FmSetB$
when $\FmSetA \; \sequent \; \FmSetB$
and $\FmSetB \; \sequent \; \FmSetA$.
The complement of a given \SetSet{} logic
$\sequent$ will be denoted
by $\notsequentx{}$.
We say that $\vdash'$
\emph{extends} $\vdash$ when
$\vdash' \; \supseteq \; \vdash$.
It is worth recalling that the collection of all extensions of
a given logic forms a complete lattice
under inclusion.
Given $\Sigma \subseteq \Sigma'$,
a logic~$\vdash'$ over $\Sigma'$ is a \emph{conservative
extension} of a logic~$\vdash$ over~$\Sigma$
when~$\vdash'$ extends~$\vdash$ and, for all $\FmSetA \cup \{\FmB\} \subseteq \LangSet{\Sigma}$, we have
$\FmSetA \vdash' \FmB$ if{f} $\FmSetA \vdash \FmB$.
These concepts may be extended to the $\SetSet$ framework in the obvious way.
We say, in addition, that a \SetFmla{} logic~$\vdash$~over $\Sigma$ 
\emph{has a disjunction} provided that 
$\FmSetA, \FmA\lor\FmB \;\vdash\; \FmC$ {iff} $\FmSetA,\FmA \;\vdash\; \FmC$ and $\FmSetA,\FmB\;\vdash\;\FmC$ (for $\lor$ a binary connective in $\Sigma$).

A \emph{(logical) $\Sigma$-matrix} $\DefMat$
is a structure $\langle \mathbf{A}, D \rangle$
where $\mathbf{A}$
is a $\Sigma$-algebra
and the members of $D \subseteq A$ are called \emph{designated values}.
We will write $\overline{D}$ to refer to $A{\setminus{}}D$.
{In case $D = A$,
we say that $\DefMat$ is \emph{trivial}.}
Provided that $\mathbf{A}$ has a lattice structure with underlying order~$\leq$, we will
often employ the notation $\Upset{a} \SymbDef \{b \in A \mid a \leq b\}$ when specifying sets of designated values.
{For instance, 
over $\ISSix{}$ we may consider the
set of designated values
$\Upset{\both} = \{\both, \tvalue, \etvalue\}$ (see~Figure~\ref{fig:pp_algebra}).}
The mappings in
$\Hom(\LangAlg{\Sigma}, \mathbf{A})$ are called \emph{$\DefMat$-valuations}. 
Every $\Sigma$-matrix
determines a \SetSet{} logic
$\sequent_\DefMat$
such that $\FmSetA \sequent{}_{\DefMat}\; \FmSetB$
if{f} $h(\FmSetA) \cap \overline{D} \neq \varnothing$
or $h(\FmSetB) \cap D \neq \varnothing$
as well as a \SetFmla{} logic
$\vdash_\DefMat$
with $\FmSetA \vdash_\DefMat \FmB$ iff $\FmSetA \sequent_\DefMat \{\FmB\}$
(notice that $\vdash_{\DefMat}$ 
is the \SetFmla{} companion of $\sequent_\DefMat$).
Given a \SetSet{} logic~$\sequent$ (resp.\ a \SetFmla{} logic~$\vdash$), 
if $\sequent\;\subseteq\;\sequent_{\DefMat}$ (resp.\ $\vdash \;\subseteq\; \vdash_{\DefMat}$), we shall say that~$\DefMat$ \emph{is a model of}
$\sequent$ (resp.\ $\vdash$),
and if the converse also holds
we shall say that~$\DefMat$ \emph{characterises}
$\sequent$ (resp.\ $\vdash$).
The $\SetSet{}$ (resp.\ $\SetFmla$) logic determined by a class $\mathcal{M}$
of $\Sigma$-matrices is given
by $\bigcap \{\sequent_\DefMat \mid \DefMat \in \mathcal{M}\}$
(resp.\ $\bigcap \{\vdash_\DefMat \mid \DefMat \in \mathcal{M}\}$).

\begin{example}
\label{ex:fde}
The $\DMSig$-matrix $\langle \FDEAlg, \Upset{\both} \rangle$ determines the logic known as
the 4-valued Dunn-Belnap logic, or First-Degree Entailment (FDE) {\rm\cite{belnap1977}},
which we hereby denote by $\FDELogic$.
Extensions of $\FDELogic$ 
are known as \emph{super-Belnap logics}  {\rm\cite{rivieccio2012}}.
\end{example}

\begin{example}
Classical Logic, hereby denoted by $\CL$, is determined by the
$\DMSig$-matrix $\langle \BoolAlg, \{\tvalue\} \rangle$.
\end{example}

Every $\Sigma$-variety
$\mathsf{K}$
such that each $\mathbf{A} \in \mathsf{K}$ has a bounded lattice reduct with greatest element $\top^\mathbf{A}$ and least element $\bot^\mathbf{A}$
induces
{a finitary \SetSet{} \emph{order-preserving logic}
$\sequent_\mathsf{K}^{\leq}$
according to which
$\FmSetB$ follows from
$\FmSetA$ if{f} 
there
exist finite 
$\FmSetA'\subseteq\FmSetA$ and
$\FmSetB'\subseteq \FmSetB$
such that the equation
$\bigwedge\FmSetA'\approx \bigwedge\FmSetA' \land \bigvee \FmSetB'$
is valid in~$\mathsf{K}$ (as usual, we assume $\bigwedge \emptyset=\top^\mathbf{A}$ and $\bigvee \emptyset=\bot^\mathbf{A}$).
}
%
The \SetFmla{} companion of $\sequent_\mathsf{K}^{\leq}$
is usually referred to as the \SetFmla{} \emph{order-preserving logic induced by~$\mathsf{K}$}, which we denote by
$\vdash_\mathsf{K}^{\leq}$.
Notice that, according to this logic, 
$\FmSetA \vdash_\mathsf{K}^{\leq} \FmB$ if, and only if,
(i)
$\FmSetA=\varnothing$
and $\FmB \approx \top$ is valid in~$\mathsf{K}$ 
or
(ii)
there
are $\FmA_1,\ldots,\FmA_n \subseteq \FmSetA$ ($n \geq 1$)
such that the equation
$\bigwedge_i \FmA_i \approx \bigwedge_i \FmA_i \land \FmB$
is valid in~$\mathsf{K}$.
Furthermore, we associate to $\mathsf{K}$
the \emph{$1$-assertional logics}
$\sequent^\top_\mathsf{K}$ and $\OneAssert_{\mathsf{K}}$ corresponding respectively to the \SetSet{} and \SetFmla{} logics determined by the class of $\Sigma$-matrices $\{\langle \mathbf{A}, \{\top^\mathbf{A}\} \rangle \mid \mathbf{A} \in \mathsf{K}\}$ {(notice that $\OneAssert_{\mathsf{K}}$ is the \SetFmla{} companion of $\sequent^\top_\mathsf{K}$)}.


A \emph{lattice filter}
of a $\bigwedge$-semilattice~$\mathbf{A}$ with a top element~$\top$ is a subset $D \subseteq A$ with $\top^\mathbf{A} \in D$
and closed under~$\land^\mathbf{A}$;
moreover,
$D$ is a \emph{proper}
lattice filter of $\mathbf{A}$ when
$D \neq A$.
If $\mathbf{A}$ is a $\bigvee$-semilattice,
a \emph{prime filter} of $\mathbf{A}$
is a proper lattice filter $D$ of $\mathbf{A}$
such that $a \lor b \in D$ iff $a \in D$
or $b \in D$, for all $a,b \in A$.
{In case every $\mathbf{A} \in \mathsf{K}$ has a bounded 
distributive 
lattice reduct,
as it happens with all varieties treated in the present work, the order-preserving logic 
induced by~$\mathsf{K}$}
coincides with the logic determined by the
class of matrices $\{\langle \mathbf{A}, D \rangle \mid \mathbf{A} \in \mathsf{K}, D \subseteq A \text{ is a non-empty lattice filter of } \mathbf{A}\}$.

Based on \cite{shoesmith_smiley:1978, marcelinowollic},
we define a \emph{symmetrical} (\emph{Hilbert-style}) \emph{calculus} $\CalcVar$ (or \emph{\SetSet{} calculus}, for short)
as a collection
of pairs $(\FmSetA,\FmSetB) \in \powerset{\LangSet{\Sigma}}\times\powerset{\LangSet{\Sigma}}$, denoted by $\inferx[]{\FmSetA}{\FmSetB}$ and called (\emph{symmetrical}) \emph{inference rules}, where
$\FmSetA$ is the
\emph{antecedent} and $\FmSetB$
is the \emph{succedent} of the said rule.
We will adopt the convention of
omitting curly braces when writing sets of formulas
and leaving a blank space instead of writing $\varnothing$
when presenting inference rules and statements
involving (generalized) consequence relations.
We proceed to define what constitutes a proof
in such calculi.

A \emph{bounded rooted tree} $\DefTree$ is a poset
$\langle \NodeSet(\DefTree), \TreeRel^\DefTree \rangle$
with a single minimal element $\Root{\DefTree}$, 
the \emph{root} of $\DefTree$,
such that, for each \emph{node} $\DefNode \in \NodeSet(\DefTree)$, the set $\{\DefNode' \in \NodeSet(\DefTree) \mid \DefNode' \TreeRel^\DefTree \DefNode \}$ of \emph{ancestors} of $\DefNode$ (or the \emph{branch} up to $\DefNode$) is
well-ordered under $\TreeRel^\DefTree$,
and every branch of $\DefTree$ has
a maximal element (a \emph{leaf} of $\DefTree$).
We may assign a label $\Label{\DefTree}(\DefNode) \in \powerset{\LangSet{\Sigma}} \cup \{\Disc\}$ to each node $\DefNode$ of $\DefTree$, in which case $\DefTree$ is said to
be \emph{labelled}. 
Given $\FmSetB \subseteq \LangSet{\Sigma}$, a leaf $\DefNode$ is \emph{$\FmSetB$-closed} in $\DefTree$ when
$\Label{\DefTree}(\DefNode) \;=\; \Disc$ or $\Label{\DefTree}(\DefNode) \cap \FmSetB \neq \varnothing$. The tree $\DefTree$ itself is
\emph{$\FmSetB$-closed} when all of its leaves are
$\FmSetB$-closed. The immediate successors of
a node $\DefNode$ with respect to $\TreeRel^\DefTree$ are called the \emph{children}
of $\DefNode$ in $\DefTree$.

Let $\mathsf{R}$ be a symmetrical calculus. 
An \emph{$\mathsf{R}$-derivation} is a labelled bounded rooted tree 
such that for every non-leaf node $\DefNode$ of $\DefTree$
there exists a rule of inference $\mathsf{r} = \inferx[]{\FmSetC}{\FmSetD} \in \mathsf{R}$ and a substitution $\sigma$
such that $\sigma(\FmSetC) \subseteq \Label{\DefTree}(\DefNode)$, and the set of the children of $\DefNode$ is 
either (i) 
$\{\DefNode^\Fm \mid \Fm \in \sigma(\FmSetD)\}$, in case $\FmSetD \neq \varnothing$, where $\DefNode^\Fm$ is a node labelled with $\Label{\DefTree}(\DefNode) \cup \{\Fm\}$, 
or
(ii) a singleton $\{\DefNode^\Disc\}$
    with $\Label{\DefTree}(\DefNode)\;=\;\Disc$,
    in case $\FmSetD\;=\;\varnothing$.
We say that
$\FmSetA\; \sequent_{\CalcVar}\; \FmSetB$
whenever there is a $\FmSetB$-closed derivation $\DefTree$ such that $\FmSetA\;\supseteq\;\Root{\DefTree}$; 
such a tree consists in a \emph{proof}
that $\FmSetB$ follows from~$\FmSetA$ in~$\mathsf{R}$.
As a matter of simplification when drawing such trees, we usually avoid copying the formulas inherited from the parent nodes (see Example~\ref{ex:fdeax} below).
The relation $\sequent_{\CalcVar}$ so defined is
a \SetSet{} logic
and, when $\sequent_{\CalcVar} = \sequent_{\DefMat}$,
we say that $\CalcVar$
\emph{axiomatizes} $\DefMat$.
A rule $\inferx[]{\FmSetA}{\FmSetB}$ is \emph{sound} with respect to $\DefMat$ when $\FmSetA \;\sequent_\DefMat\;\FmSetB$.
It should be pointed out
that such deductive formalism
generalises the conventional ($\SetFmla{}$) Hilbert-style calculi:
the latter corresponds to symmetrical calculi
whose rules have, each, a finite antecedent
and a singleton as succedent.
Given $\Lambda \subseteq \LangSet{\Sigma}$,
we write $\FmSetA \sequent_{\mathsf{R}}^\Lambda \FmSetB$
whenever there is a proof of $\FmSetB$ from $\FmSetA$
using only formulas in $\Lambda$.
We say that
$\mathsf{R}$ is \emph{$\FmSetAnalytic$-analytic}
when, for all $\FmSetA,\FmSetB \subseteq \LangSet{\Sigma}$,
whenever $\FmSetA\;\sequent_\mathsf{R}\;\FmSetB$,
we have
$\FmSetA \sequent_\mathsf{R}^{\Upsilon^\FmSetAnalytic} \FmSetB$,
with $\Upsilon^\FmSetAnalytic := \mathsf{sub}(\FmSetA \cup \FmSetB) \cup \{ \sigma(\FmA) \mid \FmA \in \FmSetAnalytic\mbox{ and } \sigma : P \to \mathsf{sub}(\FmSetA \cup \FmSetB)\}$,
which we shall dub the \emph{generalized subformulas} of $(\FmSetA, \FmSetB)$.
Intuitively, it means that
a proof in $\mathsf{R}$
that $\FmSetB$ follows from $\FmSetA$
may only use subformulas of $\FmSetA\cup\FmSetB$
or substitution instances of the
formulas in~$\FmSetAnalytic$ built with those same subformulas.

A general method is introduced in \cite{marcelinowollic, marcelino:2019} for obtaining analytic calculi (in the sense of analyticity introduced in the above paragraph)
for logics given by 
a $\Sigma$-matrix $\langle \mathbf{A}, D\rangle$ whenever a certain expressiveness requirement (called `monadicity' in \cite{shoesmith_smiley:1978}) is met: for every $a,b \in A$, there is a single-variable formula
$\mathrm{S}$ (a so-called \emph{separator}) such that $\mathrm{S}^\mathbf{A}(a) \in D$ and $\mathrm{S}^\mathbf{A}(b) \not\in D$
or vice-versa. The following example illustrates a symmetrical
calculus for $\FDELogic$ generated by this method, as well
as some proofs in this calculus.

\begin{example}
\label{ex:fdeax}
The matrix $\langle \mathbf{DM}_4, {\uparrow}\both \rangle$ fulfills the above expressiveness requirement, with the following set of separators: $\mathcal{S} \SymbDef \{p, \dmneg p\}$. 
We may therefore apply the method introduced in {\rm\cite{marcelino:2019}} to obtain for $\mathcal{B}$ the following $\mathcal{S}$-analytic
axiomatization we call~$\mathsf{R}_\mathcal{B}$:
\vspace{-.5em}
\begin{align*}
\inferx[\mathsf{r}_1]
{}
{\top}
&&
\inferx[\mathsf{r}_2]
{\dmneg \top}
{}
&&
\inferx[\mathsf{r}_3]
{}
{\dmneg \bot}
&&
\inferx[\mathsf{r}_4]
{\bot}
{}
&&
\inferx[\mathsf{r}_5]
{p}
{\dmneg \dmneg p}
&&
\inferx[\mathsf{r}_6]
{\dmneg \dmneg p}
{p}
\end{align*}
\vspace{-1.5em}
\begin{align*}
\inferx[\mathsf{r}_7]
{p \land q}
{p}
&&
\inferx[\mathsf{r}_8]
{p \land q}
{q}
&&
\inferx[\mathsf{r}_9]
{p, q}
{p \land q}
&&
\inferx[\mathsf{r}_{10}]
{\dmneg p}
{\dmneg(p \land q)}
&&
\inferx[\mathsf{r}_{11}]
{\dmneg q}
{\dmneg(p \land q)}
&&
\inferx[\mathsf{r}_{12}]
{\dmneg (p \land q)}
{\dmneg p, \dmneg q}
\end{align*}
\begin{align*}
\inferx[\mathsf{r}_{13}]
{p}
{p \lor q}
&&
\inferx[\mathsf{r}_{14}]
{q}
{p \lor q}
&&
\inferx[\mathsf{r}_{15}]
{p \lor q}
{p, q}
&&
\inferx[\mathsf{r}_{16}]
{\dmneg p, \dmneg q}
{\dmneg(p \lor q)}
&&
\inferx[\mathsf{r}_{17}]
{\dmneg(p \lor q)}
{\dmneg p}
&&
\inferx[\mathsf{r}_{18}]
{\dmneg (p \lor q)}
{\dmneg q}\\
\end{align*}

\vspace{-1em}
\noindent Figure {\rm \ref{fig:derivations}} illustrates some derivations in $\mathsf{R}_\mathcal{B}$.

\vspace{-.5em}
\begin{figure}[H]
    \centering
    \scalebox{0.9}{
    \begin{tikzpicture}[node distance=1.5cm]
    \node (0)                     {$\dmneg (p \land q)$};
    \node (11) [below left of=0]  {$\dmneg p$};
    \node (12) [below right of=0] {$\dmneg q$};
    \node (21) [below of=11]      {$\dmneg p \lor \dmneg q$};
    \node (22) [below of=12]      {$\dmneg p \lor \dmneg q$};
    \draw (0)  -- (11) node[midway,above left] {$\mathsf{r}_{12}$};
    \draw (0)  -- (12);
    \draw (11) -- (21) node[midway,left] {$\mathsf{r}_{13}$};
    \draw (12) -- (22) node[midway,right] {$\mathsf{r}_{14}$};
    \end{tikzpicture}
    \begin{tikzpicture}[node distance=1.5cm]
    \node (0)                     {$\dmneg p \lor \dmneg q$};
    \node (11) [below left of=0]  {$\dmneg p$};
    \node (12) [below right of=0] {$\dmneg q$};
    \node (21) [below of=11]      {$\dmneg (p \land q)$};
    \node (22) [below of=12]      {$\dmneg (p \land q)$};
    \draw (0)  -- (11) node[midway,above left] {$\mathsf{r}_{15}$};
    \draw (0)  -- (12);
    \draw (11) -- (21) node[midway,left] {$\mathsf{r}_{10}$};
    \draw (12) -- (22) node[midway,right] {$\mathsf{r}_{11}$};
    \end{tikzpicture}
    \qquad \qquad
    \begin{tikzpicture}[node distance=1.5cm]
    \node (0)                     {$p \lor \bot$};
    \node (11) [below left of=0]  {$p$};
    \node (12) [below right of=0] {$\bot$};
    \node (22) [below of=12]      {*};
    \draw (0)  -- (11) node[midway,above left] {$\mathsf{r}_{15}$};
    \draw (0)  -- (12);
    \draw (12) -- (22) node[midway,right] {$\mathsf{r}_{4}$};
    \node (11) [right of=12]               {$p \lor \bot$};
    \node (00) [above=0.5\distance of 11] {$p, q$};
    \node (22) [below of=11]               {};
    \draw (00) -- (11) node[midway,left] {$\mathsf{r}_{13}$};
    \end{tikzpicture}}
    \caption{Proofs in $\mathsf{R}_\FDELogic$ witnessing that $\dmneg(p \land q) \;\SymLogEquiv_{\FDELogic}\; \dmneg p \lor \dmneg q$ and $p \lor \bot \; \SymLogEquiv_{\FDELogic} \; p, q$.}
    \label{fig:derivations}
\end{figure}
\end{example}


\sloppy
Let $\Sigma$ be any signature containing a unary
connective $\DMNeg$. A $\SetSet{}$ logic $\sequent$ over $\Sigma$ is said to be
\mbox{\emph{$\DMNeg$-paraconsistent}} when
we have $\DefProp,\DMNeg\DefProp \;\notsequent\; \DeffProp$,
and \emph{$\DMNeg$-paracomplete} when
we have $\DeffProp\;\notsequent\;\DefProp,\DMNeg\DefProp$, with $\DefProp,\DeffProp \in P$.
Moreover, $\sequent$
is \emph{$\DMNeg$-gently explosive} in case there is a collection $\bigcirc(\DefProp) \subseteq \LangSet{\Sigma}$ of formulas
on a single variable such that, 
for some $\Fm \in \LangSet{\Sigma}$, we have
$\bigcirc(\Fm), \Fm \;\notsequent\; \varnothing$
and
$\bigcirc(\Fm), \DMNeg\Fm \;\notsequent\; \varnothing$,
and, for all $\Fmm \in \LangSet{\Sigma}$, we have
$\bigcirc(\Fmm),\Fmm,\DMNeg\Fmm \; \sequent \varnothing$.
Dually, $\sequent$ is \emph{$\DMNeg$-gently implosive} in case
there is a collection of formulas $\CompSet(\DefProp) \subseteq \LangSet{\Sigma}$ on a single
variable such that, for some $\Fm \in \LangSet{\Sigma}$, we have 
$\varnothing\notsequent\;\Fm,\CompSet(\Fm)$ and
$\varnothing\notsequent\;\DMNeg\Fm,\CompSet(\Fm)$,
and, for all $\Fmm \in \LangSet{\Sigma}$, we have
$\sequent\;\DMNeg\Fmm,\Fmm,\CompSet(\Fmm)$.
A \SetSet{} logic is \emph{$\DMNeg$-paradefinite} when
it is both $\DMNeg$-paraconsistent and $\DMNeg$-paracomplete; is a \emph{logic of formal inconsistency (\textbf{LFI})} when it is $\DMNeg$-paraconsistent yet $\DMNeg$-gently explosive; and is a \emph{logic of formal undeterminedness (\textbf{LFU})} when it is
$\DMNeg$-paracomplete yet $\DMNeg$-gently
implosive.
Furthermore, 
if $\sequent_1$ and $\sequent_2$
are logics over $\Sigma_1\supseteq\Sigma$
and $\Sigma_2\supseteq\Sigma$ respectively,
we say that $\sequent_1$ is a \emph{\textbf{C}-system 
based on} $\sequent_2$ with respect to $\DMNeg$ (or simply a \emph{\textbf{C}-system}) when it is an \textbf{LFI}
that
agrees with $\sequent_2$
on statements involving formulas without
$\DMNeg$ (that is, $\FmSetA \sequent_1 \FmSetB$ iff $\FmSetA \sequent_2 \FmSetB$ for all sets $\FmSetA,\FmSetB$ of formulas without $\DMNeg$), and
$\bigcirc(\DefProp) = \{\cons\DefProp\}$,
for $\cons$ a composite \emph{consistency connective} in the language of $\sequent_1$. We may dually define the notions
of \emph{\textbf{D}-system} 
{
and of \emph{determinedness connective} \cite{jmarcos2005}.  
It is worth pointing out that in the present paper we will have $\CompSet(\DefProp) = \{\DMNeg\cons\DefProp\}$.%
}

\begin{example}
By exploiting
the fact that $\neither,\both\in\FourSet$ are
fixpoints of $\DMNeg^{\FDEAlg}$,
one may easily notice that
$\FDELogic$ is $\DMNeg$-paraconsistent and $\DMNeg$-paracomplete (thus $\DMNeg$-paradefinite).
\end{example}

\section{Perfect paradefinite algebras and their logics}\label{sec:paradefinite-stone-algebras}

\subsection{Involutive Stone and PP-algebras}

We propose in this section 
to enrich De Morgan algebras by the addition of a perfection
operator $\cons$, which will allow us to recover
the classical properties of $\DMNeg$-consistency and $\DMNeg$-determinedness. In the sequel,
we will
prove that the variety
of the algebras thus obtained is term-equivalent to the variety of IS-algebras.
\vspace{.2em}

\begin{definition}\label{def:ppalgebra}
Given a $\DMoSig$-algebra whose
$\DMSig$-reduct is a De Morgan
algebra, we say that it
constitutes a \emph{perfect paradefinite algebra} (\emph{PP-algebra}) if it
satisfies the equations:
\begin{table}[H]
\begin{tabular}{@{}llll@{}}
    \textbf{(PP1)} $\cons\cons x \approx \top$
    \hspace{5mm}
    &
    \textbf{(PP2)} $\circ x \approx \circ\DMNeg x$
    \hspace{5mm}
    & 
    \textbf{(PP3)} $\cons \top \approx \top$
    \hspace{5mm}
    &
    \textbf{(PP4)} $ x \land \DMNeg x \land \cons x \approx \bot$ 
\end{tabular}
\begin{tabular}{@{}llll@{}}
    \textbf{(PP5)} $\circ(x \land y) \approx (\circ x \lor \circ y) \land (\circ x \lor \DMNeg y) \land (\circ y \lor \DMNeg x)$
\end{tabular}
\end{table}
\end{definition}
\vspace{-1em}

\begin{example}
An example of PP-algebra is $\PPAlg_6 \SymbDef \langle \SixSet, \cdot^{\PPAlg_6} \rangle$,
the $\DMoSig$-algebra
defined as $\ISAlg_6$
in Example {\rm\ref{ex:issix}},
differing only in that,
instead of containing an interpretation for
$\nabla$, it interprets $\cons$
as follows:
\[
\cons^{\PPAlg_6} a \SymbDef
\begin{cases}
    \efvalue & a \in \SixSet\setminus\{\efvalue,\etvalue\}\\
    \etvalue & a \in \{\efvalue,\etvalue\}\\
\end{cases}
\]
\noindent 
Other examples are the algebras $\PPAlg_i$, for $2 \leq i \leq 5$, the 
subalgebras
of $\PPAlg_6$
having, respectively, the same lattice structures of the algebras $\ISAlg_i$ exhibited in Figure~{\rm\ref{fig:pp_algebra}}.
\end{example}

As it occurs with IS-algebras, in the language of PP-algebras
we may  define, 
by setting $\PPComp x \;\SymbDef\; \cons x \land \DMNeg x$, 
a pseudo-complement satisfying the Stone equation.
We denote by $\PPVar$
the variety of
PP-algebras.
The following result
illustrates some useful
equations 
satisfied by the members
of $\PPVar$.
\vspace{.2em}

\begin{lemma}
\label{lem:psident}
Every PP-algebra satisfies:
\begin{tasks}[style=enumerate](3)
    \task $\DMNeg\cons x \lor (x \lor \DMNeg x) \approx \top$
    \task $\cons x \land \DMNeg\cons x \approx \bot$
    \task $\cons x \approx \cons x \land (x \lor \DMNeg x)$
\end{tasks}
\end{lemma}
\begin{proof}
Notice that 1 is a straightforward consequence
of $\EqName{(PP4)}$,
and 2 is a consequence of
$\EqName{(PP4)}$ using $\cons x$
in place of $x$ and invoking $\EqName{(PP1)}$.
Finally, 3 may be easily proved
using 1 and 2.
\end{proof}
\vspace{-.5em}
Given $\FmA \in \LangSet{\DMnSig}$
(resp.\ $\FmA \in \LangSet{\DMoSig}$),
let $\FmA^{\cons} \in \LangSet{\DMoSig}$
(resp.\ $\FmA^{\nabla} \in \LangSet{\DMnSig}$)
be the result of 
applying the definition
of $\cons$ (resp.\ of $\nabla$)
given below, in
Theorem \ref{the:interpcons} (resp.\ Theorem \ref{the:interpnabla}),
over $\FmA$.
Extend this notion
to sets of formulas
in the usual way.
The subsequent results
establish the term-equivalence between the varieties of involutive Stone algebras and
of perfect paradefinite algebras.
We first provide ways of constructing
PP-algebras from IS-algebras, and vice-versa.

\begin{theorem}
\label{the:interpcons}
\sloppy
Let
$\mathbf{A} \in \ISVar$. 
Then the $\DMoSig$-algebra $\mathbf{A}^\cons$
having the same $\DMSig$-reduct of $\mathbf{A}$ and with
$\cons^{\mathbf{A}^\cons}$ being the operation induced by $\DMNeg\nabla(x \land \DMNeg x)$ on $\mathbf{A}$
is a PP-algebra.
\end{theorem}
\vspace{-.8em}
\begin{proof}
We must check that $\mathbf{A}^{\cons}$
satisfies each of the characteristic equations of PP-algebras:
\begin{description}
    \item[(PP1)]
    $\cons\cons x \approx_{def} \DMNeg\nabla((\DMNeg\nabla(x \land \DMNeg x)) \land \DMNeg(\DMNeg\nabla(x \land \DMNeg x)))
                  \approx_{\EqName{(IS3)}} \DMNeg\nabla\DMNeg\nabla(x \land \DMNeg x) \lor \DMNeg\nabla\DMNeg\DMNeg\nabla(x \land \DMNeg x)
                  \approx_{2.5.5} \DMNeg\DMNeg\nabla(x \land \DMNeg x) \lor \DMNeg\nabla\DMNeg\DMNeg\nabla(x \land \DMNeg x)
                  \approx_{\EqName{(DM1)}} \nabla(x \land \DMNeg x) \lor \DMNeg\nabla\nabla(x \land \DMNeg x)
                  \approx_{2.5.4} \nabla(x \land \DMNeg x) \lor \DMNeg\nabla(x \land \DMNeg x)
                  \approx_{\EqName{(IS4)}} \top$.
    \item[(PP2)]
    $\cons x \approx_{def} \DMNeg\nabla(x \land \DMNeg x)
             \approx_{\EqName{(DM1)}} \DMNeg\nabla(\DMNeg\DMNeg x \land \DMNeg x)
             \approx_{def} \cons \DMNeg x$ .
    \item[(PP3)]
    $\cons \top \approx_{def} \DMNeg\nabla(\top \land \DMNeg \top)
                \approx \DMNeg\nabla(\top \land \bot)
                \approx \DMNeg\nabla\bot
                \approx_{\EqName{(IS1)}} \DMNeg\bot
                \approx \top$.
    \item[(PP4)]
    $\cons x \land (\DMNeg x \land x) \approx_{def} \DMNeg\nabla(x \land \DMNeg x) \land (\DMNeg x \land x)
                                     \approx (\DMNeg\nabla(x \land \DMNeg x) \land \DMNeg x) \land x
                                     \approx_{2.5.3} \DMNeg\nabla x \land x
                                     \approx_{2.5.2} \bot$.
    \item[(PP5)]
    $\cons(x \land y) \approx_{def} \DMNeg\nabla((x \land y) \land \DMNeg(x \land y))
                      \approx_{\EqName{(IS3)}} \DMNeg\nabla(x \land y) \lor \DMNeg\nabla\DMNeg(x \land y)
                      \approx_{\EqName{(IS3)}} (\DMNeg\nabla x \lor \DMNeg\nabla y) \lor \DMNeg\nabla\DMNeg(x \land y)
                      \approx_{2.5.6} (\DMNeg\nabla x \lor \DMNeg\nabla y) \lor (\DMNeg\nabla\DMNeg x \land \DMNeg\nabla\DMNeg y)
                      \approx (\DMNeg\nabla x \lor \DMNeg\nabla y \lor \DMNeg\nabla\DMNeg x) \land (\DMNeg\nabla x \lor \DMNeg\nabla y \lor \DMNeg\nabla\DMNeg y)
                      \approx_{2.5.3} (\DMNeg\nabla x \lor (\DMNeg\nabla(y \land \DMNeg y) \land \DMNeg y) \lor \DMNeg\nabla\DMNeg x) \land (\DMNeg\nabla y \lor (\DMNeg\nabla(x \land \DMNeg x) \land \DMNeg x) \lor \DMNeg\nabla\DMNeg y)
                      \approx_{\EqName{(IS3)}} (\DMNeg\nabla (x \land \DMNeg x) \lor (\DMNeg\nabla(y \land \DMNeg y) \land \DMNeg y)) \land (\DMNeg\nabla (y \land \DMNeg y) \lor (\DMNeg\nabla(x \land \DMNeg x) \land \DMNeg x))
                      \approx_{def} (\cons x \lor (\cons y \land \neg y)) \land (\cons y \lor (\cons x \land \neg x))
                      \approx (\cons x \lor \cons y) \land (\cons x \lor \DMNeg y) \land (\cons y \lor \DMNeg x)$.
\qedhere

\end{description}
\end{proof}
\vspace{-.5em}

\begin{theorem}
\label{the:interpnabla}
Let $\mathbf{A} \in \PPVar$.
Then the $\DMnSig$-algebra $\mathbf{A}^\nabla$
having the same $\DMSig$-reduct of $\mathbf{A}$ and with
$\nabla^{\DMISExtend{\mathbf{A}}}$
being the operation induced by
$\DMNeg\cons x \lor x$ on $\mathbf{A}$
is an IS-algebra.
\end{theorem}
\vspace{-.8em}
\begin{proof}
We must check that $\mathbf{A}^\nabla$
satisfies each of the characteristic equations of IS-algebras:
\begin{description}
    \item[(IS1)]
    $\nabla \bot \approx_{def} \DMNeg\cons \bot \lor \bot
                 \approx \DMNeg\cons\bot
                 \approx \DMNeg\cons\DMNeg\top
                 \approx_{\EqName{(PP2)}} \DMNeg\cons\top
                 \approx_{\EqName{(PP3)}} \DMNeg\top
                 \approx \bot$.
    \item[(IS2)] By absorption and commutativity of $\lor$, we have
    $x \land \nabla x \approx_{def} x \land (\DMNeg\cons x \lor x)
                      \approx x$.
    \item[(IS3)]
    $\nabla(x \land y) \approx_{def} \DMNeg\cons(x \land y) \lor (x \land y)
                       \approx_{\EqName{(PP5)}} (\DMNeg\circ x \land \DMNeg\circ y) \lor (\DMNeg\circ x \land y) \lor (\DMNeg\circ y \land  x) \lor (x \land y)
                       \approx (\DMNeg\cons x \lor x) \land (\DMNeg\cons y \lor y)
                       \approx_{def} \nabla x \land \nabla y$.
    \item[(IS4)]
    $\DMNeg\nabla x \land \nabla x \approx_{def} \DMNeg(\DMNeg\cons x \lor x) \land (\DMNeg\cons x \lor x)
                                  \approx_{\EqName{(DM2)}} (\cons x \land \DMNeg x) \land (\DMNeg\cons x \lor x)
                                  \approx (\cons x \land \DMNeg x \land \DMNeg\cons x) \lor (\cons x \land \DMNeg x \land x)
                                  \approx_{\EqName{(PP4)}} (\cons x \land \DMNeg x \land \DMNeg\cons x) \lor \bot
                                  \approx \cons x \land \DMNeg x \land \DMNeg\cons x
                                  \approx \cons x \land \DMNeg x \land \DMNeg\cons x \land \top
                                  \approx_{\EqName{(PP1)}} \cons x \land \DMNeg x \land \DMNeg\cons x \land \cons\cons x
                                  \approx_{\EqName{(PP4)}} \bot \land \DMNeg x
                                  \approx \bot$.
\qedhere
\end{description}
\vspace{-1em}
\end{proof}

Then, for the announced term-equivalence, we just need to check that:

\begin{theorem}
\label{the:equiv}
Given $\mathbf{A}\in\ISVar$
and $\mathbf{B}\in\PPVar$,
we have $\left(\mathbf{A}^\cons\right)^\nabla = \mathbf{A}$
and $\left(\mathbf{B}^\nabla\right)^\cons = \mathbf{B}$.
\end{theorem}
\vspace{-.8em}
\begin{proof}
In order to prove that
$\left(\mathbf{A}^\cons\right)^\nabla = \mathbf{A}$, it is enough
to show that $\DMNeg(\DMNeg\nabla(x \land \DMNeg x))\lor x \approx \nabla x$ holds in~$\mathbf{A}$, that
is, the operation induced by 
the term $\left((\nabla x)^\cons\right)^\nabla$ 
coincides with the interpretation of~$\nabla$.
By the fact that $\nabla x \lor x \approx \nabla x$, we have
$\DMNeg(\DMNeg\nabla(x \land \DMNeg x)) \lor x \approx_{\EqName{(DM1)}} \nabla(x \land \DMNeg x) \lor x
                                            \approx_{\EqName{(IS3)}} (\nabla x \land \nabla \DMNeg x) \lor x
                                            \approx (\nabla x \lor x) \land (\nabla \DMNeg x \lor x)
                                            \approx_{{\rm Lemma}\ 2.5.1} (\nabla x \lor x) \land \top
                                            \approx \nabla x \lor x
                                            \approx \nabla x$.
Similarly, for proving $\left(\mathbf{B}^\nabla\right)^\cons = \mathbf{B}$, it is enough to
show that 
$\left((\cons x)^\nabla\right)^\cons$
induces an operation that coincides with the interpretation of~$\cons$, which amounts to proving
that $\DMNeg(\DMNeg\cons(x \land \DMNeg x) \lor (x \land \DMNeg x)) \approx \cons x$
holds in $\mathbf{B}$.
Then, we have
$\DMNeg(\DMNeg\cons(x \land \DMNeg x) \lor (x \land \DMNeg x)) \approx_{\EqName{(DM2)}} \cons(x \land \DMNeg x) \land (\DMNeg x \lor x)
                                                           \approx_{\EqName{(PP5)}} (\cons x \lor \cons\DMNeg x) \land (\cons x \lor x) \land (\cons\DMNeg x \lor \DMNeg x) \land (\DMNeg x \lor x)
                                                           \approx_{\EqName{(PP2)}} (\cons x \lor \cons x) \land (\cons x \lor x) \land (\cons x \lor \DMNeg x) \land (\DMNeg x \lor x) 
                                                           \approx
                                                           \cons x \land (\DMNeg x \lor x)
                                                           \approx_{{\rm Lemma}\ 3.3.3} \cons x$.
\end{proof}



By inspecting the interpretation
induced by the definition of
$\cons$ in terms of
$\nabla$ given in Theorem \ref{the:interpcons},
one may easily check the following result.

\begin{proposition}
$\PPAlg_i = \DMPPExtend{\ISAlg}_i$, for all $2 \leq i \leq 6$.
\end{proposition}

From the equivalence just presented
and a similar result
for IS-algebras \cite{marcelino2021}, we may now conclude that the
variety of PP-algebras
is generated by $\PPSix$:
\begin{proposition}
\label{prop:sixgenerated}
$\PPVar = \Variety(\{\PPSix\})$.
\end{proposition}

\subsection{
Logics associated to PP-algebras}

{Recall that we denote by $\PPlogicMC{}$ and $\PPlogic$, respectively, the \SetSet{} and \SetFmla{} order-preserving
logics induced by $\PPVar$.
Also, we denote by $\PPOnelogicMC{}$ and $\PPOnelogic{}$, respectively,
the \SetSet{} and \SetFmla{} 1-assertional logics induced by $\PPVar$.}
We will use the following auxiliary results 
together with analogous
results for $\ISlogic$ \cite{marcelino2021}
{(which smoothly generalizes to $\ISlogicMC{}$, the \SetSet{} order-preserving logic associated to $\ISVar$)}
to prove some characterizations of
the logics associated to $\PPVar$
in terms of single finite logical matrices.

\begin{lemma}
\label{lem:homomorphisms-relations}
Given $\mathbf{A} \in \ISVar$
and $\mathbf{B} \in \PPVar$,
\begin{enumerate}\itemsep0pt
    \item if 
    $h$ is a valuation on $\mathbf{A}$,
    then $h\left(\left(\Fm^\cons\right)^\nabla\right) = h(\Fm)$
    for all $\Fm \in \LangSet{\DMnSig}$;
    \item if 
    $h$ is a valuation on $\mathbf{B}$,
    then $h\left(\left(\Fm^\nabla\right)^\cons\right) = h(\Fm)$
    for all $\Fm \in \LangSet{\DMoSig}$;
    \item if 
    $h$ is a valuation on $\mathbf{A}$,
    then 
    the mapping $h^\cons \in \Hom(\LangAlg{\DMoSig}, \mathbf{A}^\cons)$
    such that $h^\cons(p) = h(p)$
    for all $p \in P$
    satisfies $h^\cons(\Fm^\circ) = h(\Fm)$ for all $\Fm \in \LangSet{\DMnSig}$;
    \item if 
    $h$ is a valuation on $\mathbf{B}$,
    then 
    the mapping $h^\nabla \in \Hom(\LangAlg{\DMnSig}, \mathbf{B}^\nabla)$
    such that $h^\nabla(p) = h(p)$
    for all $p \in P$
    satisfies $h^\nabla(\Fm^\nabla) = h(\Fm)$ for all $\Fm \in \LangSet{\DMoSig}$.
\end{enumerate}
\end{lemma}
\vspace{-.8em}
\begin{proof}
We will first discuss the proofs of items 1 and 3, 
which may then be easily adapted, respectively, for proving items 2 and 4.
Both proofs are by structural induction on the set of formulas. Starting with 1, when $\Fm \in P$,
the result trivially
holds, as propositional variables are not affected
by translations.
In case $\Fm = \nabla \FmB$, 
if $h((\FmB^\cons)^\nabla) = h(\FmB)$, we will have
$h((\Fm^\cons)^\nabla)) = h(((\nabla\FmB)^\cons)^\nabla)$. 
From the argument in the proof of Theorem~\ref{the:equiv},
we know that $((\nabla\FmB)^\cons)^\nabla$
and $\nabla\FmB$ induce the
same operation on $\mathbf{A}$,
thus $h(((\nabla\FmB)^\cons)^\nabla) = \nabla(h(\FmB)) = h(\Fm)$.
The proof is analogous for
the cases of $\land,\lor,\DMNeg,\top$ and $\bot$. 
Now, for item 3, the base case is again obvious,
and, in case $\Fm = \nabla \FmB$,
we have $h^\cons((\nabla\FmB)^\cons)=h^\cons(\DMNeg\cons\FmB^\cons \lor \FmB^\cons)=\DMNeg^{\mathbf{A}^\cons}\cons^{\mathbf{A}^\cons}h^\cons(\FmB^\cons) \lor^{\mathbf{A}^\cons}h^\cons(\FmB^\cons)$,
and, by the induction hypothesis,
the latter is equal to
$\DMNeg^{\mathbf{A}^\cons}\cons^{\mathbf{A}^\cons}h(\FmB) \lor^{\mathbf{A}^\cons}h(\FmB)$;
this is the same as $h(\nabla\FmB)$
in $(\mathbf{A}^\cons)^\nabla$, which coincides
with $\mathbf{A}$ by Theorem \ref{the:equiv}.
The proof is again analogous for
$\land,\lor,\DMNeg,\top$ and~$\bot$.
\end{proof}

The former result allows us to prove the following auxiliary facts:

\begin{proposition}~
\label{prop:logictransl}%
\begin{enumerate}\itemsep0pt
    {\item $\FmSetA \sequent_{\langle \mathbf{A}, D \rangle} \FmSetB$ \; if{f} \; $\FmSetA^\nabla \sequent_{\langle \mathbf{A}^\nabla, D \rangle} \FmSetB^\nabla$,
    where $\mathbf{A} \in \PPVar$ and $\langle \mathbf{A}, D\rangle$ is a $\DMoSig$-matrix\vspace{.2em}
    \item $\FmSetA \sequent_\mathcal{M} \FmSetB$ \; if{f} \;
    $\FmSetA^\nabla \sequent_{\mathcal{M}^\nabla} \FmSetB^\nabla$, for $\mathcal{M} = \{\langle \mathbf{A}, D \rangle \mid \mathbf{A} \in \PPVar\}$
    \item $\FmSetA \;\sequent_{\PPlogic}\; \FmSetB$ \; if{f} \; $\FmSetA^\nabla \;\sequent_{\ISlogic}\; \FmSetB^\nabla$
    \item $\FmSetA \; {\sequent^\top_\PPVar} \; \FmSetB$ \; if{f} \; $\FmSetA^\nabla\; {\sequent^\top_\ISVar}\; \FmSetB^\nabla$}
\end{enumerate}
\end{proposition}
\vspace{-.5em}
\begin{proof}
We start by proving item 1. From the left to the right, suppose that
    there is a valuation
    $h \in \Hom\left(\LangAlg{\DMnSig}, \mathbf{A}^\nabla\right)$ such that $h(\FmSetA^\nabla) \subseteq D$ while $h(\FmSetB^\nabla) \subseteq \overline{D}$.
    By items 2 and 3 of Lemma \ref{lem:homomorphisms-relations}, there is a
    valuation $h^\cons \in \Hom(\LangAlg{\DMoSig}, (\mathbf{A}^\nabla)^\cons) = \Hom(\LangAlg{\DMoSig}, \mathbf{A})$
    such that $h^\cons((\FmSetA^\nabla)^\cons) = h^\cons(\FmSetA)$
    and $h^\cons((\FmSetA^\nabla)^\cons) = h(\FmSetA^\nabla)$,
    thus $h^\cons(\FmSetA) = h(\FmSetA^\nabla) \subseteq D$. 
    Similarly, we may conclude that
    $h^\cons(\FmSetB) \subseteq \overline{D}$,
    and we are done.
    The other direction
    is similar, but using item 4 of
    Lemma~\ref{lem:homomorphisms-relations}. 
    Item 2, above, is a clear consequence of
    item~1, and items 3 and 4 follow directly from items 1 and 2, respectively.
\end{proof}

From this fact, we obtain that the order-preserving logics $\PPlogicMC{}$ and $\PPlogic{}$
are determined by a single 6-valued logical matrix:

\begin{theorem}
\label{the:order-preserving-six-generated}
$\PPlogicMC{} \;=\;\sequent_{\langle \PPSix, \Upset{\both} \rangle}$,
and thus $\PPlogic{} \;=\;\vdash_{\langle \PPSix, \Upset{\both} \rangle}$.
\end{theorem}
\vspace{-.5em}
\begin{proof}
By Proposition \ref{prop:logictransl}
and the fact that $\sequent_{\ISlogic}$ is
characterized by the matrix 
\mbox{$\langle \ISSix, {\uparrow}{\both} \rangle$}~\cite{marcelino2021},
we have
$\FmSetA \sequent_{\langle \PPSix, \uparrow\both \rangle} \FmSetB$ if{f} $\FmSetA^\nabla \sequent_{\langle \ISSix, \uparrow\both \rangle} \FmSetB^\nabla$
if{f} $\FmSetA^\nabla \sequent_{\ISlogic} \FmSetB^\nabla$
if{f} $\FmSetA \sequent_{\PPlogic} \FmSetB$.
\end{proof}

Furthermore, we have that the 1-assertional logics $\PPOnelogicMC{}$ and $\PPOnelogic{}$
are determined by a single 3-valued matrix:

\begin{proposition}
$\PPOnelogicMC{} \;=\; \sequent^{\top}_{\Variety(\PPAlg_3)} \;=\; \sequent_{\langle \PPAlg_3, \{\etvalue\} \rangle}$,
and thus $\PPOnelogic{} \;=\; \OneAssert_{\Variety(\PPAlg_3)} \;=\; \vdash_{\langle \PPAlg_3, \{\etvalue\} \rangle}$.
\end{proposition}
\begin{proof}
It is clear that $\sequent^{\top}_\PPVar \;\subseteq\; \sequent^{\top}_{\Variety(\PPAlg_3)} \;\subseteq\; \sequent_{\langle \PPAlg_3, \{\etvalue\} \rangle}$.
The result then follows because $\sequent^{\top}_\PPVar \;=\; \sequent_{\langle \PPAlg_3, \{\etvalue\} \rangle}$,
as $\FmSetA \sequent_{\langle \PPAlg_3, \{\etvalue\} \rangle} \FmSetB$ if{f} 
$\FmSetA^\nabla \sequent_{\langle \PPAlg_3^\nabla, \{\etvalue\} \rangle} \FmSetB^\nabla$ (by Lemma \ref{lem:homomorphisms-relations}) iff $\FmSetA^\nabla \sequent_{\langle \mathbf{IS}_3, \{\etvalue\} \rangle} \FmSetB^\nabla$ (because $\mathbf{IS}_3 = \PPAlg_3^\nabla$) if{f}
$\FmSetA^\nabla \sequent^{\top}_{\ISVar} \FmSetB^\nabla$ (by \cite[Prop.\ 4.5]{marcelino2021}) if{f}
$\FmSetA \sequent^{\top}_{\PPVar} \FmSetB$ (by Lemma \ref{lem:homomorphisms-relations}).
%
\end{proof}

As the last item in this series of characterizations, we
have, as it should be expected, that the
1-assertional
logic associated to $\{\PPAlg_2\}$ coincides with Classical Logic:

\begin{proposition}
\label{prop:clpp2}
{For all $\FmSetA, \FmSetB \subseteq \LangSet{\DMSig}$,
$\FmSetA \sequent_\CL \FmSetB \;\text{ iff }\; \FmSetA \sequent^\top_{\PPAlg_2} \FmSetB$}.
\end{proposition}
\begin{proof}
The result follows from the clear isomorphism between
$\PPAlg_2$ and $\BoolAlg$.
\end{proof}


Finally, we may explore the term-equivalence presented in the previous subsection (Theorem~\ref{the:equiv}) to
prove another important fact about $\PPlogic$.
For the definitions of full self-extensionality, protoalgebraizability and algebraizability that appear in the following result, we refer the reader
to \cite[Definitions 5.25, 6.1 and 3.11, resp.]{font:2016}.

\begin{proposition}
\label{the:one-assert-three-generated}
$\PPlogic$ is 
fully
self-extensional 
and non-protoalgebraic (hence non-algebraizable).
\end{proposition}
\begin{proof}
The result follows from \cite[Prop.\ 4.2]{marcelino2021},
\cite[Theorem 7.18, item 4]{font:2016}, and
the term-equivalence of $\ISVar$ with $\PPVar$
given by Theorem \ref{the:equiv}.
\end{proof}


\subsection{De Morgan algebras with a perfection operator}

We now present a  recipe for constructing a perfect paradefinite algebra by endowing a De Morgan algebra with a perfection operator. 
This should be of particular interest, as we shall see in subsection~\ref{subsec:rec}, for an investigation on \textbf{LFI}s and \textbf{LFU}s when the De Morgan algebra at hand happens not to be Boolean.
We will see in the next section how to axiomatize logics induced by PP-algebras produced through this recipe, starting from a calculus for the logic induced by a De Morgan algebra given as input.
\vspace{.5em}
\begin{definition}
\label{def:expcons}
Let $\mathbf{A}$
be a $\DMSig$-algebra. Given $\aefvalue, \aetvalue \notin A$, we define the $\DMoSig$-algebra $\mathbf{A}^\cons \SymbDef \langle A \cup \{\aefvalue, \aetvalue\}, \cdot^{\mathbf{A}^\cons} \rangle$ by letting:
\vspace{-.5em}
\begin{align*}
    a \land^{\mathbf{A}^\cons} b &\SymbDef
    \begin{cases}
        a \land^\mathbf{A} b & \text{if } a,b \in A\\
        \aetvalue & \text{if } a = b = \aetvalue\\
        \aefvalue & \text{if } a = \aefvalue \text{ or }  b = \aefvalue\\
        c & \text{if } \{a,b\} = \{\aetvalue, c\} \text{ with } c \in A
    \end{cases}
    &
    a \lor^{\mathbf{A}^\cons} b &\SymbDef
    \begin{cases}
        a \lor^\mathbf{A} b & \text{if } a,b \in A\\
        \aefvalue & \text{if } a = b = \aefvalue\\
        \aetvalue & \text{if } a = \aetvalue \text{ or } b = \aetvalue\\
        c & \text{if } \{a,b\} = \{\aefvalue, c\} \text{ with } c \in A
    \end{cases}
    \\
    \DMNeg^{\mathbf{A}^\cons} a &\SymbDef
    \begin{cases}
        \DMNeg^\mathbf{A} a & \text{if } a \in A\\
        \aefvalue & \text{if } a = \aetvalue\\
        \aetvalue & \text{if } a = \aefvalue
    \end{cases}
    &
    \cons^{\mathbf{A}^\cons} a &\SymbDef
    \begin{cases}
        \aetvalue & \text{if } a = \aefvalue \text{ or } a = \aetvalue\\
        \aefvalue & \text{otherwise}
    \end{cases}
    \\
    \bot^{\mathbf{A}^\cons} &\SymbDef \aefvalue
    &
    \top^{\mathbf{A}^\cons} &\SymbDef \aetvalue
\end{align*}
\noindent 
In addition, we define
the $\DMnSig$-algebra $\DMISExtend{\mathbf{A}}\SymbDef \langle A \cup \{\aefvalue, \aetvalue\}, \cdot^{\mathbf{A}^\nabla} \rangle$
interpreting the connectives in $\DMSig$ as above,
while letting $\nabla^{\DMISExtend{\mathbf{A}}} a \SymbDef \efvalue$ if $a = \efvalue$ and $\nabla^{\DMISExtend{\mathbf{A}}} a \SymbDef \etvalue$ otherwise (cf. {\rm\cite{marcelino2021}}).
\end{definition}

\begin{proposition}
\label{prop:psfromdm}
If
$\mathbf{A}$ is a De Morgan algebra,
then
$\mathbf{A}^\cons$ is
a PP-algebra.
\end{proposition}
\begin{proof}
When $\mathbf{A}$ is a De Morgan algebra,
it is clear that the $\DMSig$-reduct of $\mathbf{A}^\cons$ is also a De Morgan algebra.
Moreover, the operation $\cons^{\mathbf{A}^\cons}$ defined above
satisfies all equations presented in Definition \ref{def:ppalgebra}, as we confirm below:
\begin{description}
    \item[(PP1)]
    By the definition of $\cons^{\mathbf{A}^\cons}$, we have either
    (1) $\cons^{\mathbf{A}^\cons} a = \aetvalue$ or (2) $\cons^{\mathbf{A}^\cons} a = \aefvalue$.
    In both cases we have $\cons^{\mathbf{A}^\cons} \cons^{\mathbf{A}^\cons} a = \aetvalue = \top^{\mathbf{A}^\cons}$.
    \item[(PP2)]
    By the definition of $\DMNeg^{\mathbf{A}^\cons}$, we have that
    $a \in \{\aefvalue, \aetvalue\}$ iff $\DMNeg^{\mathbf{A}^\cons} a \in \{\aefvalue, \aetvalue\}$.
    Also, we have either (1) $a \in \{\aefvalue, \aetvalue\}$ or (2) $a \notin \{\aefvalue, \aetvalue\}$.  
    If (1) is the case, then $\cons^{\mathbf{A}^\cons} a = \cons^{\mathbf{A}^\cons} \DMNeg^{\mathbf{A}^\cons} a = \aetvalue$; 
    alternatively, if (2) is the case, then $\cons^{\mathbf{A}^\cons} a = \cons^{\mathbf{A}^\cons} \DMNeg^{\mathbf{A}^\cons} a = \aefvalue$.
    \item[(PP3)]
    By the definition of $\cons^{\mathbf{A}^\cons}$ and $\top^{\mathbf{A}^\cons}$, we have that
    $\cons^{\mathbf{A}^\cons} \top^{\mathbf{A}^\cons} = \cons^{\mathbf{A}^\cons} \aetvalue = \aetvalue = \top^{\mathbf{A}^\cons}$.
    \item[(PP4)]
    By the definition of $\cons^{\mathbf{A}^\cons}$, we have either
    (1) $\cons^{\mathbf{A}^\cons} a = \aetvalue$ or (2) $\cons^{\mathbf{A}^\cons} a = \aefvalue$.
    If (1) is the case, we have either (1.1) $a = \aefvalue$ or (1.2) $a = \aetvalue$, then:
    If (1.2) is the case, then, by the definition of $\DMNeg^{\mathbf{A}^\cons}$, $\DMNeg^{\mathbf{A}^\cons} a = \aefvalue$.
    In all cases we have that at least one among $a, \DMNeg^{\mathbf{A}^\cons} a$ and $\cons^{\mathbf{A}^\cons} a$ is $\aefvalue$. Then, by the definition of $\land^{\mathbf{A}^\cons}$, we have $a \land \DMNeg^{\mathbf{A}^\cons} a \land^{\mathbf{A}^\cons} \cons^{\mathbf{A}^\cons} a = \aefvalue = \bot^{\mathbf{A}^\cons}$.
    \item[(PP5)]
    By the definition of $\cons^{\mathbf{A}^\cons}$, we have that either
    (1) $\cons^{\mathbf{A}^\cons} (a \land^{\mathbf{A}^\cons} b) = \aetvalue$ or (2) $\cons^{\mathbf{A}^\cons} (a \land^{\mathbf{A}^\cons} b) = \aefvalue$.
    If (1) is the case, we have either (1.1) $a \land^{\mathbf{A}^\cons} b = \aefvalue$ or (1.2) $a \land^{\mathbf{A}^\cons} b = \aetvalue$. 
    If (1.1) is the case, then, by the definition of $\land^{\mathbf{A}^\cons}$, we have either (1.1.1) $a = \aefvalue$ or (1.1.2) $b = \aefvalue$.
    If (1.1.1) is the case, then, by the definition of $\cons^{\mathbf{A}^\cons}$ and $\dmneg^{\mathbf{A}^\cons}$, we have both $\cons^{\mathbf{A}^\cons} a = \aetvalue$ and $\DMNeg^{\mathbf{A}^\cons} a = \aetvalue$.
    The case (1.1.2) is similar to (1.1.1).
    If (1.2) is the case, then, by the definition of $\land^{\mathbf{A}^\cons}$, we have $a = b = \aetvalue$.
    By the definition of $\cons^{\mathbf{A}^\cons}$, we have $\cons^{\mathbf{A}^\cons} a = \cons^{\mathbf{A}^\cons} b = \aetvalue$.
    Hence, in all subcases of (1) we have, by the definition of $\lor^{\mathbf{A}^\cons}$:
    $\cons^{\mathbf{A}^\cons}a \lor^{\mathbf{A}^\cons} \cons^{\mathbf{A}^\cons}b = \cons^{\mathbf{A}^\cons} a \lor^{\mathbf{A}^\cons} \DMNeg^{\mathbf{A}^\cons} b = \cons^{\mathbf{A}^\cons} b \lor \DMNeg^{\mathbf{A}^\cons} a = \aetvalue$.
    If (2) is the case, we have, $a \land^{\mathbf{A}^\cons} b \notin \{\aefvalue, \aetvalue\}$.
    Then, by the definition of $\land^{\mathbf{A}^\cons}$, we have either (2.1) $a, b \notin \{\aefvalue, \aetvalue\}$ or (2.2) both $a = \aetvalue$ and $b \not\in \{\aefvalue, \aetvalue\}$ or (2.3) both $b = \aetvalue$ and $a \not\in \{\aefvalue, \aetvalue\}$.
    If (2.1) is the case, then, by the definition of $\cons^{\mathbf{A}^\cons}$, we have $\cons^{\mathbf{A}^\cons} a = \cons^{\mathbf{A}^\cons} b = \aefvalue$.
    If (2.2) is the case, then, by the definitions of $\cons^{\mathbf{A}^\cons}$ and $\dmneg^{\mathbf{A}^\cons}$, we have $\cons^{\mathbf{A}^\cons} b = \aefvalue$ and $\DMNeg^{\mathbf{A}^\cons} a = \aefvalue$.
    If (2.3) is the case, then, by the definitions of $\cons^{\mathbf{A}^\cons}$ and $\dmneg^{\mathbf{A}^\cons}$, we have $\cons^{\mathbf{A}^\cons} a = \aefvalue$ and $\DMNeg^{\mathbf{A}^\cons} b = \aefvalue$.
    Hence, in all subcases of (2) we have, by the definition of $\lor^{\mathbf{A}^\cons}$, that at least one among 
    $\cons^{\mathbf{A}^\cons}a \lor^{\mathbf{A}^\cons} \cons^{\mathbf{A}^\cons}b$, $\cons^{\mathbf{A}^\cons} a \lor^{\mathbf{A}^\cons} \DMNeg^{\mathbf{A}^\cons} b$ or $\cons^{\mathbf{A}^\cons} b \lor \DMNeg^{\mathbf{A}^\cons} a$ is $\aefvalue$.
    In all cases we have that:
    $\cons^{\mathbf{A}^\cons} (a \land^{\mathbf{A}^\cons} b) = (\cons^{\mathbf{A}^\cons}a \lor^{\mathbf{A}^\cons} \cons^{\mathbf{A}^\cons}b) \land (\cons^{\mathbf{A}^\cons} a \lor^{\mathbf{A}^\cons} \DMNeg^{\mathbf{A}^\cons} b) \land (\cons^{\mathbf{A}^\cons} b \lor \DMNeg^{\mathbf{A}^\cons} a)$.
    \qedhere
\end{description}
\end{proof}

\begin{example}
Comparing Figure {\rm\ref{fig:de_morgan}}
with Figure {\rm\ref{fig:pp_algebra}},
we see that $\ISAlg_6$, $\ISAlg_5$ 
and $\ISAlg_4$ coincide, respectively, with
$\DMPPExtend{\FDEAlg}$, $\DMPPExtend{\KleeneAlg}$
and $\DMPPExtend{\BoolAlg}$.
\end{example}

%
%


\subsection{The lattice of extensions of $\PPlogic$}
\label{lattice-extensions}

Given a $\DMSig$-matrix $\DefMat \SymbDef \langle \mathbf{A}, D \rangle$, let $\DMPPExtend{\DefMat} \SymbDef \langle \DMPPExtend{\mathbf{A}}, D \cup \{\aetvalue\} \rangle$ be the $\DMoSig$-matrix with the underlying
(by Proposition \ref{prop:psfromdm}, perfect paradefinite)
algebra $\DMPPExtend{\mathbf{A}}$ given by Definition \ref{def:expcons}. We denote by $\DMReductPP{\DefMat}$ the $\DMSig$-reduct of $\DMPPExtend{\DefMat}$.
Furthermore, given a class of $\DMSig$-matrices $\mathcal{M}$, we let $\DMPPExtend{\mathcal{M}} \SymbDef \{\DefMat^\cons \mid \DefMat \in \mathcal{M}\}$ and $\DMReductPP{\mathcal{M}} \SymbDef \{\DMReductPP{\DefMat} \mid \DefMat \in \mathcal{M}\}$.
 {Whenever $\SuperBelnapCon$ is a super-Belnap logic, denote by $\SuperBelnapCon^\cons$ the logic determined by the family of matrices
$\{\DMPPExtend{\DefMat} \mid \DefMat \text{ is a that non-trivial model of } \SuperBelnapCon\}$.
The series of results presented in
this section shows that the mapping $(\cdot)^\circ$ just defined
constitutes a lattice embedding from the
lattice of super-Belnap logics into the lattice
of extensions of the logic $\PPlogic$.
This allows us, in particular, to
lift the result on the lower bound of the cardinality of extensions of $\FDELogic$
to a corresponding one on the lower bound of the cardinality of extensions of
$\PPlogic$.
}

{Before introducing the results,
we recall some helpful
definitions from abstract algebraic logic~\cite{font:2016}}.
Given a $\Sigma$-matrix
$\DefMat = \langle \mathbf{A}, D \rangle$,
a congruence $\theta \in \ConSet \mathbf{A}$
is said to be \emph{compatible with} $\DefMat$
when $b \in D$ whenever both $a \in D$
and $a \theta b$, for all $a,b \in A$.
We denote by $\LeibCong{\DefMat}$
the \emph{Leibniz congruence}
associated to $\DefMat$,
namely
the greatest congruence of 
$\mathbf{A}$ compatible with $\DefMat$.
The matrix $\Reduce{\DefMat} = \langle \mathbf{A}/\LeibCong{\DefMat}, D/\LeibCong{\DefMat} \rangle$
is the \emph{reduced version} of $\DefMat$.
{We say that $\DefMat$ is
\emph{reduced} when it coincides with its own reduced version
(or, equivalently, when its Leibniz
congruence is the identity relation
on~$A$).}
It is well known that
$\sequent_\DefMat = \sequent_{\Reduce{\DefMat}}$
(and thus $\vdash_\DefMat \;=\; \vdash_{\Reduce{\DefMat}}$)
and, since every logic is determined by a class of matrix models, we have that every logic coincides
with the logic determined by
its reduced matrix models.
The class of all reduced
matrix models for a logic $\vdash$
is denoted by $\ReducedMatrices(\vdash)$.

\begin{lemma}
Let $\DefMat$ be a 
{non-trivial model  
  of
$\FDELogic$}. Then $\Reduce{\DefMat} \cong \Reduce{\left(\DMReductPP{\DefMat}\right)}$.
\end{lemma}
\begin{proof}
We know from \cite[Lemma 4.6]{marcelino2021}
that $\Reduce{\DefMat} \cong \Reduce{\left(\DMReductIS{\DefMat}\right)}$.
Clearly, $\DMReductIS{\DefMat}$
and $\DMReductPP{\DefMat}$ are isomorphic
matrices under the identity mapping on $A \cup \{\aefvalue, \aetvalue\}$, and so are their reductions.
\end{proof}

\begin{corollary}
\label{coro:redequalscons}
Where $\DefMat$ is a 
non-trivial model  
  of
$\FDELogic$, we
have $\sequent_{\DefMat} = \sequent_{\DMReductPP{\DefMat}}$ and $\vdash_{\DefMat} \;=\; \vdash_{\DMReductPP{\DefMat}}$.
\end{corollary}

\begin{corollary}
\label{coro:conserv}
Where $\DefMat$ is a 
non-trivial model  
  of
$\FDELogic$, we have that $ \sequent_{\DMPPExtend{\DefMat}}$ 
is a conservative extension of $\sequent_{\DefMat}$ and
$\vdash_{\DMPPExtend{\DefMat}}$
is a conservative extension of $\vdash_{\DefMat}$.
\end{corollary}


{\begin{corollary}\label{coro:conservsuper}
Let $\SuperBelnapCon$ be a super-Belnap logic
determined by a class of {non-trivial} models of $\FDELogic$.
Then $\SuperBelnapCon^\cons$ is a conservative
extension of $\SuperBelnapCon$.
In particular,
$\PPlogic$ is a conservative extension of $\FDELogic$.
\end{corollary}
\begin{proof}
It follows from Corollary~\ref{coro:conserv}
and the fact that $\PPlogic$ is characterized
by the matrix $\langle \PPAlg_6, \Upset{\both} \rangle$,
which is obtained from the matrix $\langle \FDEAlg, \Upset{\both} \rangle$ by the construction introduced in
Definition~\ref{def:expcons}.
\end{proof}
}


%

 

{\begin{corollary}
Let $\DefMat_1$
and $\DefMat_2$
be 
non-trivial
models of $\FDELogic$. 
If
$\sequent_{\DMPPExtend{\DefMat}_1} \;=\; \sequent_{\DMPPExtend{\DefMat}_2}$, then
$\sequent_{\DefMat_1} \,=\; \sequent_{\DefMat_2}$;
analogously,
if $\vdash_{\DMPPExtend{\DefMat}_1} \;=\; \vdash_{\DMPPExtend{\DefMat}_2}$, then
$\vdash_{\DefMat_1} \,=\; \vdash_{\DefMat_2}$.
\end{corollary}
}



\begin{corollary}
\label{coro:embdfunc}
Let $\SuperBelnapCon_1$ and $\SuperBelnapCon_2$
be super-Belnap logics. Then
$\SuperBelnapCon_1 \; \subseteq \; \SuperBelnapCon_2$
 iff $\SuperBelnapCon_1^\cons \; \subseteq \; \SuperBelnapCon_2^\cons$.   
\end{corollary}
\begin{proof}
From the left to the right, assuming
$\SuperBelnapCon_1 \;\subseteq\; \SuperBelnapCon_2$
gives that $\ReducedMatrices(\SuperBelnapCon_2) \subseteq \ReducedMatrices(\SuperBelnapCon_1)$, so $(\ReducedMatrices(\SuperBelnapCon_2))^\cons \subseteq (\ReducedMatrices(\SuperBelnapCon_1))^\cons$, which clearly entails that $\SuperBelnapCon_1^\cons \; \subseteq \; \SuperBelnapCon_2^\cons$.
Conversely, suppose that $\SuperBelnapCon_1^\cons \; \subseteq \; \SuperBelnapCon_2^\cons$
and that $\FmSetA \SuperBelnapCon_1 \FmB$.
Hence $\FmSetA \SuperBelnapCon_1^\cons \FmB$, and then
$\FmSetA \SuperBelnapCon_2^\cons \FmB$, which 
gives $\FmSetA \SuperBelnapCon_2 \FmB$
by Corollary \ref{coro:conservsuper}.
\end{proof}

\begin{corollary}
\label{coro:embdsupbelnap}
The map given by $\SuperBelnapCon \; \mapsto \; \SuperBelnapCon^\cons$ is an embedding (that is, an injective homomorphism) of the
lattice of super-Belnap logics into the lattice
of extensions of $\PPlogic$. 
This, the latter lattice
has (at least) the cardinality of the
continuum.
\end{corollary}
\begin{proof}
By Corollary \ref{coro:embdfunc} and
\cite[Theorem 4.13]{rivieccio2012}.
\end{proof}


\vspace{-2mm}
\subsection{
On the recovery of classical reasoning}
\label{subsec:rec}

The following result shows
that paradefinite extensions of
$\FDELogic$, when extended with
$\cons$ in the way we propose,
result in logics which are at once \textbf{C}-systems and
\textbf{D}-systems. This result applies,
in particular, to the logic $\PPlogic$.

\begin{proposition}
Let $\mathcal{M}$ be a class
of 
{non-trivial models}
of $\FDELogic$
that determines a paradefinite logic. Then
the \SetSet{} logic determined by
$\DMPPExtend{\mathcal{M}}$ is
a \textbf C-system and
a \textbf D-system.
\end{proposition}
\begin{proof}
That paradefiniteness is preserved when passing from
$\mathcal{M}$ to $\DMPPExtend{\mathcal{M}}$ follows
by Corollary \ref{coro:conserv}.
As it is well-known that the negation-free fragments of $\CL$
and $\FDELogic$ coincide,
by taking $\cons$ as the consistency connective
and 
$\DMNeg\cons$ 
as the determinedness connective,
we may straightforwardly use the values
$\aefvalue$ and $\aetvalue$ to build suitable valuations for showing that
the logic determined by~$\DMPPExtend{\mathcal{M}}$
is at once a \textbf C-system and a \textbf D-system.\qedhere
\end{proof}

\begin{corollary}
$\PPlogic$ is a \textbf{C}-system
and a \textbf{D}-system.
\end{corollary}

A unary connective $\copyright$ is said to constitute
a \emph{classical negation} in
a \SetFmla{} logic~$\vdash$ based on~$\Sigma$
when, for all $\FmA,\FmB \in \LangSet{\Sigma}$, we have that
(i): $\FmSetA, \FmA \vdash \FmB$
and $\FmSetA, \copyright(\FmA) \vdash \FmB$
imply $\FmSetA \vdash \FmB$,
and (ii): $\FmA,\copyright(\FmA)\vdash\FmB$.
In case $\vdash$ has a disjunction,
we may equivalently replace (i) by (iii):
$\varnothing \vdash \FmA \lor \copyright(\FmA)$ in this characterization. We prove in what follows that in $\PPlogic$ no composite unary connective may be defined that simultaneously satisfies both (i) and (iii). Since $\PPlogic$ has a disjunction, this entails that a classical negation is not definable in this logic.

\begin{proposition}
\label{prop:no_classical_neg}
There is no unary formula $\FmA \in \LangSet{\DMoSig}$ such that
$p, \FmA(p) \vdash_{\PPlogic} q$
and \mbox{$\varnothing \vdash_{\PPlogic} p \lor \FmA(p)$}.
\end{proposition}
\begin{proof}
Let $\FmA \in \LangSet{\DMoSig}$ be
a unary formula
and
suppose that
$p, \FmA(p) \vdash_{\PPlogic} q$
and $\varnothing \vdash_{\PPlogic} p \lor \FmA(p)$. Then,
since $\PPlogic$ is an order-preserving logic,
we have, for all
$h \in \Hom(\LangAlg{\DMoSig}, \PPSix)$, that
$h(p) \lor^{\PPSix} \FmA^{\PPSix}(h(p)) = \etvalue$ 
(the greatest element of $\PPSix$)
and $h(p) \land^{\PPSix} \FmA^{\PPSix}(h(p)) = \efvalue$
(the least element of $\PPSix$)
, which is to say that
$\FmA^\PPSix(a)$ is a Boolean complement of
$a$, for every element $a$
of $\PPSix$. This is absurd, since, by the definition of $\land^\PPSix$ and $\lor^\PPSix$, only $\etvalue$
and $\efvalue$ have Boolean complements in $\PPSix$.
\end{proof}

On what concerns the previous result, it is worth observing that a similar phenomenon, concerning the undefinability of a classical negation, is observed concerning several \textbf{LFI}s and \textbf{LFU}s with a modal character (\cite[Theorem 6.1.2]{lahav2017}) built on top of complete distributive lattices.

As argued in \cite{jmarcos2005}, the ability to recover
negation-consistent (resp.\ negation-determined) reasoning is the most fundamental feature
of \textbf{LFI}s (resp.\ \textbf{LFU}s). 
This feature may be expressed
in terms of a convenient \emph{Derivability Adjustment Theorem}
(DAT) with respect to Classical Logic, which states, in the present case, that
classical reasoning may be fully recovered as long as premises restoring the lost `perfection' and establishing the `classicality' of a certain set of formulas are available. 
The result presented below is a DAT that applies to any
super-Belnap logic
determined by a class of {non-trivial} models of $\FDELogic$
extended with the perfection operator $\cons$ considered in this paper. As a corollary, we will, in particular, have a DAT for the logics $\PPlogicMC{}$ and $\PPlogic$.

\begin{theorem}
\label{the:datgeneral}
Let $\mathcal{M}$ be a class of 
{non-trivial}
models of
$\FDELogic$. Then, for all $\FmSetA, \FmSetB \subseteq \LangSet{\DMSig}$, we have
\[
\FmSetA \;\sequent_\CL\; \FmSetB \; \text{ if{f} } \; \FmSetA, \cons \DefProp_1,\ldots,\cons\DefProp_n \;\sequent_{\DMPPExtend{\mathcal{M}}}\; \FmSetB,
\]
with $\{\DefProp_1,\ldots,\DefProp_n\} = \Props(\FmSetA \cup \FmSetB)$.
\end{theorem}
\begin{proof}
Let $\mathcal{M}$ be a class of 
{non-trivial}
models of $\FDELogic$.
Notice that $\langle \PPAlg_2, \{\etvalue\} \rangle$ is a submatrix of $\DMPPExtend{\DefMat}$
for all
$\DMPPExtend{\DefMat} \in \DMPPExtend{\mathcal{M}}$.

From the left to the right, contrapositively,
suppose that $\FmSetA, \cons \DefProp_1,\ldots,\cons\DefProp_n\; \notsequent_{\DMPPExtend{\mathcal{M}}}\; \FmSetB$. Then, there
are $\DMPPExtend{\DefMat} = \langle \DMPPExtend{\mathbf{A}}, D \cup \{\etvalue\} \rangle \in \DMPPExtend{\mathcal{M}}$ and $h \in \Hom(\LangAlg{\DMoSig}, \DMPPExtend{\mathbf{A}})$ such
that (a) $h(\FmSetA \cup \{\cons \DefProp_1,\ldots,\cons\DefProp_n\}) \subseteq D \cup \{\etvalue\}$
and (b) $h(\FmSetB) \subseteq \overline{D} \cup \{\efvalue\}$.
The interpretation of $\cons$ given in Definition \ref{def:expcons} and (a) entail that $h(\DefProp_i) \in \{\efvalue, \etvalue\}$ for all
$1 \leq i \leq n$. As $\PPAlg_2$ is a subalgebra of
$\DMPPExtend{\mathbf{A}}$, we may define
an $h' : \{\DefProp_1,\ldots,\DefProp_n\} \to \{\efvalue, \etvalue\}$
by setting $h'(\DefProp_i) := h(\DefProp_i)$; this extends to the full language and, in view of Definition \ref{def:expcons}, agrees with $h$
on the set $\FmSetA \cup \FmSetB$. Thus, by (a), $h'(\FmSetA) \subseteq \{\etvalue\}$ (as $\efvalue \not\in D$), while $h'(\FmSetB) \subseteq \{\efvalue\}$ by (b), meaning that $\FmSetA \notsequent^\top_{\PPAlg_2} \FmSetB$. Hence, by Proposition \ref{prop:clpp2}, we have
$\FmSetA \;\notsequent_{\CL}\; \FmSetB$.

From the right to the left, again contrapositively, assume that $\FmSetA\; \notsequent_{\CL}\; \FmSetB$. Thus, by Proposition \ref{prop:clpp2}, we have
$\FmSetA\; \notsequent^\top_{\PPAlg_2}\; \FmSetB$. Then there is $h \in \Hom(\LangAlg{\DMoSig}, \PPAlg_2)$ such that $h(\FmSetA) \subseteq \{\etvalue\}$ and $h(\FmSetB) \subseteq \{\efvalue\}$. Notice that, if
 $\DMPPExtend{\DefMat} = \langle \DMPPExtend{\mathbf{A}}, D \cup \{\etvalue\} \rangle \in \DMPPExtend{\mathcal{M}}$,
 then we may define $h' : P \to A \cup \{\efvalue, \etvalue\}$
with $h'(\DefProp) = h(\DefProp)$, for all $\DefProp \in P$.
As $\PPAlg_2$ is a subalgebra of $\DMPPExtend{\mathbf{A}}$,
$h'$ extends to the full language and agrees with $h$ on
it. Moreover,
as $h'(\DefProp_i) \in \{\efvalue, \etvalue\}$, we have,
by Definition \ref{def:expcons}, $h'(\cons \DefProp_i) = \etvalue$,
for all $1 \leq i \leq n$. Hence $h'(\FmSetA \cup\{\cons\DefProp_1,\ldots,\cons\DefProp_n\}) \subseteq \{\etvalue\}$,
while $h'(\FmSetB) \subseteq \{\efvalue\}$. Therefore, $\FmSetA, \cons\DefProp_1,\ldots,\cons\DefProp_n \;\notsequent_{\DMPPExtend{\DefMat}} \;\FmSetB$ for each
$\DMPPExtend{\DefMat} \in \DMPPExtend{\mathcal{M}}$, and,
in particular, we obtain $\FmSetA, \cons\DefProp_1,\ldots,\cons\DefProp_n \;\notsequent_{\DMPPExtend{\mathcal{M}}} \; \FmSetB$.
\end{proof}

{\begin{corollary}
\label{coro:datpp}
For all $\FmSetA,\FmSetB \subseteq \LangSet{\DMSig}$, we have
\[
\FmSetA \sequent_\CL \FmSetB \; \text{ if{f} } \; \FmSetA, \cons \DefProp_1,\ldots,\cons\DefProp_n \sequent^{\leq}_{\PPVar} \FmSetB,
\]
with $\{\DefProp_1,\ldots,\DefProp_n\} = \Props(\FmSetA \cup \FmSetB)$.
\end{corollary}
\begin{proof}
Follows by Theorem~\ref{the:datgeneral}, together with the facts that $\FDEAlg$ is 
{non-trivial}
and that $\PPlogicMC$ is determined
by the single matrix $\PPAlg_6$, which coincides
with $\DMPPExtend{\FDEAlg}$.
\end{proof}
}

\section{Axiomatizing Logics of De Morgan Algebras Enriched with Perfection}
\label{sec:axiomatization}

In the first part of this section, we provide
a general recipe for producing a
symmetrical Hilbert-style calculus
for the \SetSet{} logic determined
by any class $\mathcal{M}$ of $\DMSig$-matrices
expanded with the perfection operator~$\cons$ according to the mechanism set up in the previous section.
Our approach is based on adding some rules governing $\cons$ to
a given axiomatization of~$\mathcal{M}$, resulting in what we call
a \emph{relative axiomatization} of $\mathcal{M}^\cons$ by the added rules with
respect to the \SetSet{} logic determined by~$\mathcal{M}$. 
In the sequel, we
will show, for a particular class of matrices, how to turn the given \SetSet{} relative axiomatizations into \SetFmla{} axiomatizations,
using the fact proved in \cite[Theorem 5.37]{shoesmith_smiley:1978}
that a symmetrical calculus $\mathsf{R}$
can be transformed into a \SetFmla{}
calculus provided that 
$\vdash_\mathsf{R}$ has a disjunction.
If $\mathsf{R}$ axiomatizes a class
$\mathcal{M}$ of $\DMSig$-matrices,
a sufficient condition for the latter property to hold is that all members of $\mathcal{M}$ have prime filters as sets of designated values. 
{For this reason, the provided \SetFmla{}
Hilbert-style calculi will consist in axiomatizations for logics determined by classes of $\DMSig$-matrices whose designated values form prime filters.}

\subsection{Analyticity-preserving symmetrical calculi}
In what follows, 
if $\sequent_1$ and $\sequent_2$
are \SetSet{} logics over
$\Sigma$,
we set $\sequent_1 \simeq \sequent_2$
if{f} $\sequent_1 \cup \{(\LangSet{\Sigma}, \varnothing)\} = \sequent_2 \cup \{(\LangSet{\Sigma}, \varnothing)\}$.
It is clear that two \SetSet{} logics satisfying this
condition induce the same $\SetFmla{}$ logic.\footnote{This has been observed by R.\ Carnap, already in the 1940s \cite{Carnap43}.  It might seem that extending a logic this way would imply that a semantics characterising the extended logic would have to provide `models for contradictory formulas'. However, such a model, in this case, would be trivial, for it would make all formulas equally true. As argued in \cite{marcos-ineffable}, this is not the kind of models that a paraconsistent logician is interested upon.  This explains, by the way, why our definition of paraconsistency, presented towards the end of Section \ref{sec:preliminares}, has been formulated in terms of $\DefProp,\DMNeg\DefProp \;\notsequent\; \DeffProp$ rather than $\DefProp,\DMNeg\DefProp \;\notsequent\; \varnothing$.}
We will employ this weaker relation instead of the equality relation to make the results in this section more general and simpler to prove.
The first result below provides a generic recipe for axiomatizing the \SetSet{}
logic determined by the class $\mathcal{M}^\cons$, assuming we have a calculus $\mathsf{R}$ that axiomatizes the
\SetSet{} logic determined by $\DMReductPP{\mathcal{M}}$ {(namely, the family of the $\DMSig$-reducts of the matrices in $\mathcal{M}^\cons$)}.
The rules listed
in this result 
are among those obtained by
running the axiomatization
algorithm described in~\cite{marcelino:2019}
on the matrix $\langle \PPAlg_6, \Upset{\mathbf{b}} \rangle$,
using $\{\DefProp, \cons\DefProp, \DMNeg\DefProp\}$
as set of separators,
and then streamlining the resulting
calculus. What we will see now is that $\mathsf{R}$ together with these very rules axiomatizes
the $\SetSet{}$
logic determined by $\mathcal{M}^\cons$.

\begin{theorem}\label{thm:4.1}
\label{the:expcalculus}
Let $\mathcal{M}$ be a class of $\DMSig$-matrices. If $\sequentx{\DMReductPP{\mathcal{M}}} \simeq \sequentx{\mathsf{R}}$, then $\sequentx{\mathcal{M}^\cons} = \sequentx{\mathsf{R} \cup \mathsf{R}_\cons}$, where $\mathsf{R}_\cons$ consists of the following inference rules:
\vspace{-.5em}
\begin{align*}
\inferx[\mathsf{r}_1]
{}
{\cons \bot}
&&
\inferx[\mathsf{r}_2]
{}
{\cons \top}
&&
\inferx[\mathsf{r}_3]
{}
{\cons \cons p}
&&
\inferx[\mathsf{r}_4]
{\cons p }
{\cons \DMNeg p}
&&
\inferx[\mathsf{r}_5]
{\cons \DMNeg p }
{\cons p}
&&
\inferx[\mathsf{r}_6]
{\cons p}
{p, \DMNeg p}
&&
\inferx[\mathsf{r}_7]
{\cons p, p, \DMNeg p}
{}
\end{align*}
\vspace{-1.5em}
\begin{align*}
\inferx[\mathsf{r}_8]
{\cons p}
{\cons(p \land q), p}
&&
\inferx[\mathsf{r}_9]
{\cons q}
{\cons(p \land q), q}
&&
\inferx[\mathsf{r}_{10}]
{\cons(p \land q), q}
{\cons p}
&&
\inferx[\mathsf{r}_{11}]
{\cons(p \land q), p}
{\cons q}
&&
\inferx[\mathsf{r}_{12}]
{\cons p, \cons q}
{\cons(p \land q)}
&&
\inferx[\mathsf{r}_{13}]
{\cons(p \land q)}
{\cons p, \cons q}
\end{align*}
\vspace{-1.5em}
\begin{align*}
\inferx[\mathsf{r}_{14}]
{\cons p, \cons q}
{\cons(p \lor q)}
&&
\inferx[\mathsf{r}_{15}]
{\cons(p \lor q)}
{\cons p, \cons q}
&&
\inferx[\mathsf{r}_{16}]
{\cons p, p}
{\cons(p \lor q)}
&&
\inferx[\mathsf{r}_{17}]
{\cons q, q}
{\cons(p \lor q)}
&&
\inferx[\mathsf{r}_{18}]
{\cons(p \lor q)}
{\cons p, q}
&&
\inferx[\mathsf{r}_{19}]
{\cons(p \lor q)}
{\cons q, p}
\end{align*}
\end{theorem}

\begin{proof}
Checking the soundness of those rules is routine; we provide only a couple of examples.
Let~$v$ be an $\DefMat^\cons$-valuation.
The rule $\mathsf{r}_3$ is sound in $\DefMat^\cons$, given that
$v(\cons \FmA) \in \{\aefvalue, \aetvalue\}$, 
so we have that $v(\cons \cons \FmA) = \aetvalue$. 
On what concerns rule $\mathsf{r}_8$, we have that,
if $v(\cons \FmA) = \aetvalue$, then either (i) $v(\FmA) = \aefvalue$ or
(ii) $v(\FmA) = \aetvalue$. 
Soundness is obvious in case (ii).
In case (i), $v(\FmA \land \FmB) = \aefvalue$, so $v(\cons (\FmA \land \FmB)) = \aetvalue$.

For completeness, assume $\FmSetA \;\notsequentx{\mathsf{R}_\cons}\; \FmSetB$. Then, by cut for sets, there is a partition $\langle T, F \rangle$ of $\LangSet{\DMoSig}$ such that $\FmSetA \subseteq T$ and $\FmSetB \subseteq F$ and $T \;\notsequentx{\mathsf{R}_\cons}\; F$. Note that (by $\mathsf{r}_3, \mathsf{r}_6$ and $\mathsf{r}_7$) for each $\FmA$, we have either $\cons \FmA \in T$ or $\DMNeg \cons \FmA \in T$, but never both. In particular, $F$ is never empty. Also, by $\mathsf{r}_6$ and $\mathsf{r}_7$, if we have $\cons \FmA \in T$, we have either $\FmA \in T$ or $\DMNeg \FmA \in T$, but never both. Hence, each~$\FmA$ must belong to exactly one of three cases: (a) $\DMNeg \cons \FmA \in T$, (b) $\cons \FmA, \FmA \in T$ or (c) $\cons \FmA, \DMNeg \FmA \in T$.

Since $\mathsf{R} \subseteq \mathsf{R} \cup \mathsf{R}_\cons$, we also have $T \notsequentx{\mathsf{R}} F$.
From the fact that $\sequentx{\mathsf{R}} \simeq \sequentx{\widehat{\DefMat}}$ and $F \neq \varnothing$ we know that $T \notsequentx  {\widehat{\DefMat}} F$.
We may therefore pick some $\widehat{\DefMat}$-valuation~$v$,
for some $\DefMat \in \mathcal{M}$, such that $v(T) \subseteq D$ and 
{$v(F)\subseteq\overline{D}$}.
Consider now the mapping $v^\prime : \LangSet{\DMoSig} \to \DefMat^\cons$ defined by:
\begin{equation*}
    v^\prime(\FmA) :=
    \begin{cases}
        v(\FmA_i) & \text{if } \FmA = \FmA_1 \land \FmA_2 \text{ and } \DMNeg\cons\FmA, \cons\FmA_{3-i} \in T \text{ for } i \in \{1,2\}\\
        v(\FmA_i) & \text{if } \FmA = \FmA_1 \lor \FmA_2 \text{ and } \DMNeg\cons\FmA, \cons\FmA_{i} \in T \text{ for } i \in \{1,2\}\\
        v(\FmA) & \text{if } \DMNeg\cons\FmA \in T\\
        \aetvalue & \text{if } \cons\FmA, \FmA\in T\\
        \aefvalue & \text{if } \cons\FmA, \DMNeg\FmA\in T
    \end{cases}
\end{equation*}
We will check that $v^\prime$ is an $\DefMat^\cons$-valuation:
\begin{enumerate}
    \item $v^\prime(\cons\FmA) = \cons v^\prime(\FmA)$: If (i) $\DMNeg \cons \FmA \in T$
    then, by $\mathsf{r}_3$, $\cons \cons \FmA \in T$ (so $v^\prime(\cons \FmA) = \aefvalue$).
    Thus $v^\prime(\cons \FmA) = \aefvalue = \cons(v^\prime(\FmA))$.
    If (ii) $\cons \FmA, \FmA \in T$,
    then, by $\mathsf{r}_3$, $\cons \cons \FmA \in T$ (so $v^\prime(\cons \FmA) = \aetvalue$).
    So $v^\prime(\cons \FmA) = \aetvalue = \cons(v^\prime(\FmA))$.
    Case (iii) is analogous to (ii).
    
    \item $v^\prime(\DMNeg\FmA) = \DMNeg v^\prime(\FmA)$: If (i) $\DMNeg \cons \DMNeg \FmA \in T$,
    then, by $\mathsf{r}_3$ and $\mathsf{r}_7$, $\cons \DMNeg \FmA \notin T$.
    Then, by $\mathsf{r}_4$, $\cons \FmA \notin T$.
    Thus, by $\mathsf{r}_3$ and $\mathsf{r}_6$, $\DMNeg \cons \FmA \in T$ (so $v^\prime(\FmA) = v(\FmA)$).
    So $v^\prime(\DMNeg \FmA) = v(\DMNeg \FmA) = \DMNeg v(\FmA) = \DMNeg(v^\prime(\FmA))$.
    If (ii) $\cons \DMNeg \FmA, \DMNeg \FmA \in T$,
    by $\mathsf{r}_5$, $\cons \FmA \in T$ (so $v^\prime(\FmA) = \aefvalue)$).
    Then $v^\prime(\DMNeg \FmA) = \aetvalue = \DMNeg(v^\prime(\FmA))$.
    Case (iii) is analogous to (ii).
    
    \item $v^\prime(\FmA \land \FmB) = v^\prime(\FmA) \land v^\prime(\FmB)$: If (i) $\DMNeg\cons(\FmA \land \FmB) \in T$,
    then, by $\mathsf{r}_3$ and $\mathsf{r}_7$, we have that $\cons(\FmA \land \FmB) \notin T$.
    By $\mathsf{r}_{12}$, we have that (a) $\cons \FmA, \cons \FmB \notin T$, (b) $\cons \FmA \in T$ and $\cons \FmB \notin T$ or (c) $\cons \FmA \notin T$ and $\cons \FmB \in T$.  So:
    \begin{enumerate}\itemsep0pt
        \item 
        By $\mathsf{r}_3$ and $\mathsf{r}_6$, $\DMNeg \cons \FmA, \DMNeg \cons \FmB \in T$
        (so $v^\prime(\FmA) = v(\FmA)$ and $v^\prime(\FmB) = v(\FmB)$).
        So $v^\prime(\FmA \land \FmB) = v(\FmA \land \FmB) = v(\FmA) \land v(\FmB) = v^\prime(\FmA) \land v^\prime(\FmB)$.
        \item 
        By $\mathsf{r}_3$ and $\mathsf{r}_6$, $\DMNeg\cons\FmB \in T$ (so $v'(\FmB) = v(\FmB)$).
        By $\mathsf{r}_8$, $\FmA \in T$ (so $v'(\FmA)=\aetvalue$).
        Therefore $v'(\FmA \land \FmB) = v(\FmB) = v'(\FmB) = \aetvalue \land v'(\FmB) = v'(\FmA) \land v'(\FmB)$.
        \item 
        This case is analogous to the previous one, but now using $\mathsf{r}_9$.
    \end{enumerate}
    If (ii) $\cons(\FmA \land \FmB), \FmA \land \FmB \in T$,
    then $\FmA, \FmB \in T$.
    By $\mathsf{r}_{10}$ and $\mathsf{r}_{11}$, $\cons \FmA, \cons \FmB \in T$.
    (so $v^\prime(\FmA) = v^\prime(\FmB) = \aetvalue$)
    hence $v^\prime(\FmA \land \FmB) = \aetvalue = v^\prime(\FmA) \land v^\prime(\FmB)$. If (iii) $\cons(\FmA \land \FmB), \DMNeg(\FmA \land \FmB) \in T$,
    then either $\DMNeg \FmA \in T$ or $\DMNeg \FmB \in T$.
    By $\mathsf{r}_{13}$, we have that (a) $\cons \FmA, \cons \FmB \in T$, (b) $\cons \FmA \in T$ and $\cons \FmB \notin T$ or (c) $\cons \FmA \notin T$ and $\cons \FmB \in T$.  So:
    \begin{enumerate}\itemsep0pt
        \item 
        Here, we have that 
        $v^\prime(\FmA) = \aefvalue$ or $v^\prime(\FmB) = \aefvalue$.
        So $v^\prime(\FmA \land \FmB) = \aefvalue = v^\prime(\FmA) \land v^\prime(\FmB)$.
        \item 
        By $\mathsf{r}_{11}$ and $\mathsf{r}_6$ $\DMNeg \FmA \in T$ (so $v^\prime(\FmA) = \aefvalue)$.
        So $v^\prime(\FmA \land \FmB) = \aefvalue = v^\prime(\FmA) \land v^\prime(\FmB)$.
        \item 
        This case is analogous to the previous one, using $\mathsf{r}_{10}$.
    \end{enumerate}
    \item $v^\prime(\FmA \lor \FmB) = v^\prime(\FmA) \lor v^\prime(\FmB)$: analogous to the case of $\land$.
    \item $v^\prime(\bot) = \aefvalue$ and $v^\prime(\top) = \aetvalue$: directly from rules $\mathsf{r}_1$ and $\mathsf{r}_2$.
\qedhere
\end{enumerate}
\end{proof}

Given $\FmSetAnalytic \subseteq \LangSet{\DMoSig}$, let $\FmSetAnalytic^\cons := \FmSetAnalytic \cup \{\dmneg p, \cons p, \dmneg \cons p\}$. The theorem below shows that the recipe presented above preserves analyticity.

\begin{theorem}\label{thm:4.8}
Let $\mathcal{M}$ be a class of $\DMSig$-matrices.
If $\mathsf{R}$ is a $\FmSetAnalytic$-analytic axiomatization of $\sequent_{\DMReductPP{\mathcal{M}}}$,
then $\mathsf{R} \cup \mathsf{R}_\cons$ is a $\FmSetAnalytic^\cons$-analytic axiomatization of $\sequent_{\mathcal{M}^\cons}$.
\end{theorem}
\begin{proof}
Let $\Upsilon := \mathsf{sub}(\FmSetA \cup \FmSetB)$ and $\Lambda := \Upsilon_{\FmSetAnalytic^\cons}$.
Assume that $\FmSetA \;\notsequent_{\mathsf{R} \cup \mathsf{R_\cons}}^\Lambda \; \FmSetB$.
Then, by cut for sets, there is a partition $\langle T, F \rangle$ of $\Lambda$ such that $\FmSetA \subseteq T$ and $\FmSetB \subseteq F$ and $T \notsequent_{\mathsf{R} \cup \mathsf{R_\cons}}^\Lambda F$. Since $\mathsf{R} \subseteq \mathsf{R} \cup \mathsf{R}_\cons$, we also have 
$T \notsequent_{\mathsf{R}}^\Lambda F$.
From the fact that $\mathsf{R}$ axiomatizes $\sequentx{\DMReductPP{\mathcal{M}}}$ and $F \neq \varnothing$ we know that $T \notsequentx  {\DMReductPP{\mathcal{M}}} F$.
We may therefore pick $v \in \Hom_{\DMSig}(\LangAlg{\DMoSig}, \DMReductPP{\DefMat})$, for some $\DefMat \in \mathcal{M}$, such that $v(T) \subseteq D$ and 
{$v(F)\subseteq\overline{D}$}.
Since, for each $\FmA \in \Upsilon$, we have $\dmneg \FmA, \cons \FmA, \dmneg \cons \FmA \in \Lambda$,
we may use the same construction given in Theorem \ref{thm:4.1} to define a certain mapping $v^\prime: \Upsilon \to \DefMat^\cons$.
That $v^\prime$ respects all the connectives follows from the fact that in the proof of Theorem \ref{thm:4.1}
we only used instances of the rules employing formulas present in $\Lambda$.
This, together with the fact that $\Upsilon$ is closed under subformulas, implies that $v^\prime$ is a partial $\DefMat^\cons$-valuation.
Hence, $v^\prime$ may be extended to a total $\DefMat^\cons$-valuation,
witnessing the fact that $\FmSetA\; \notsequentx{\mathcal{M}^\cons}\; \FmSetB$, thus concluding
the proof.
\end{proof}

From Theorem \ref{def:expcons} and Theorem \ref{thm:4.8}, it follows that:

\begin{corollary}
Let $\mathcal{S} \SymbDef \{p, \DMNeg p\}$.
The calculus presented
in Example {\rm\ref{ex:fdeax}}
together with
the rules of $\mathsf{R}_\cons$
is an $\mathcal{S}^\cons$-analytic axiomatization of $\PPlogic$.
\end{corollary}

As explained in \cite{marcelino:2019},
analytic calculi as those we have been discussing are associated to a proof-search algorithm and a countermodel-search algorithm,
and consequently to a decision procedure for the corresponding \SetSet{} logics. 
Briefly put, if we want to know whether $\FmSetA \;\sequent_\mathsf{R}\; \FmSetB$,
where $\mathsf{R}$ is a $\FmSetAnalytic$-analytic
symmetrical calculus, 
obtaining a proof when the answer is positive and a countermodel otherwise,
we may attempt to build a derivation in
the following way: start from
a single node labelled with $\FmSetA$
and search for a rule instance
of $\mathsf{R}$ not used in the same branch
with formulas in the set
$\Upsilon^{\FmSetAnalytic}$
(namely, the set of generalized subformulas
of $(\FmSetA,\FmSetB)$, as defined
in Section~\ref{sec:preliminares})
whose premises are in~$\FmSetA$. If there is one,
expand that node by creating a
child node labelled with $\FmSetA\cup\{\Fm\}$ for each formula $\Fm$ 
in the succedent of the chosen rule instance
and repeat this step for each new node.
In case it fails in finding a rule instance for applying to some node,
we may conclude that no proof exists,
and from each non-$\FmSetB$-closed branch we may extract a countermodel.
In case every branch eventually gets $\FmSetB$-closed, the resulting
tree is a proof of the desired statement. The following example illustrates
how this works.

\vspace{-.5em}
\begin{figure}[bth]
    \centering
    \scalebox{.9}{
    \begin{tikzpicture}[node distance=1.5cm, squarednode/.style={rectangle, draw=red!60, fill=red!5, very thick, minimum size=5mm}]
    \node (00)                     {$\varnothing$};
    \node (0)  [below=0.56\distance of 00]  {$\cons \cons p$};
    \node (11) [below left  of=0]  {$\cons p$};
    \node (12) [below right of=0]  {$\dmneg \cons p$};
    \node (21) [below left  of=11] {$p$};
    \node (22) [below right of=11] {$\dmneg p$};
    \draw (0)  -- (00) node[midway,left]       {$\mathsf{r}_3$};
    \draw (0)  -- (11) node[midway,above left] {$\mathsf{r}_6$};
    \draw (0)  -- (12);
    \draw (11) -- (21) node[midway,above left] {$\mathsf{r}_{6}$};
    \draw (11) -- (22);
    \node (x) [right of=00] {};
    \node (y) [right of=x] {};
    \node (z) [right of=y] {};
    \node (0000) [right of=z]                    {$\varnothing$};
    \node (000)  [below=0.56\distance of 0000]  {$\cons \cons p$};
    \node (110) [below left  of=000]  {$\cons p$};
    \node (120) [below right of=000]  {$\dmneg \cons p$};
    \node (210) [below=0.56\distance of 110]  {$\cons \dmneg p$};
    \node[squarednode] (310) [below left  of=210] {$\dmneg p$};
    \node (320) [below right of=210] {$\dmneg \dmneg p$};
    \node (420) [below=0.56\distance of 320]  {$p$};
    \draw (000)  -- (0000) node[midway,left]       {$\mathsf{r}_3$};
    \draw (000)  -- (110) node[midway,above left] {$\mathsf{r}_6$};
    \draw (000)  -- (120);
    \draw (110) -- (210) node[midway,left]       {$\mathsf{r}_{4}$};
    \draw (210) -- (310) node[midway,above left] {$\mathsf{r}_{6}$};
    \draw (210) -- (320);
    \draw (320) -- (420);
    \node (xx) [right of=0000] {};
    \node (yy) [right of=xx] {};
    \node (zz) [right of=yy] {};
    \node (a)  [right of=zz]     {$\varnothing$};
    \node (b)  [below=0.56\distance of a]  {$\cons \cons p$};
    \node (c) [below left  of=b]  {$\cons p$};
    \node[squarednode] (d) [below right of=b]  {$\dmneg \cons p$};
    \node (e) [below left  of=c] {$p$};
    \node (f) [below right of=c] {$\dmneg p$};
    \node (g) [below=0.56\distance of f] {$\cons p \land \dmneg p$};
    \draw (b)  -- (a) node[midway,left]       {$\mathsf{r}_3$};
    \draw (b)  -- (c) node[midway,above left] {$\mathsf{r}_6$};
    \draw (b)  -- (d);
    \draw (c) -- (e) node[midway,above left] {$\mathsf{r}_{6}$};
    \draw (c) -- (f);
    \draw (f) -- (g);
    \end{tikzpicture}}
    \vspace{-.5em}
    \caption{\\ Outputs of the proof-search and of the countermodel-search algorithm induced by our analytic symme\-trical calculus, witnessing that $\varnothing\;\sequent_{\mathsf{R}\cup\mathsf{R}^\cons}\;\DefProp,\DMNeg\DefProp,\DMNeg\cons\DefProp$; that $\varnothing\;\notsequent_{\mathsf{R}_{\FDELogic}^{} \cup \mathsf{R}_\FDELogic^\cons}\; \DefProp,\DMNeg\cons\DefProp$ and that $\varnothing\;\notsequent_{\mathsf{R}_{\FDELogic}^{} \cup \mathsf{R}_\FDELogic^\cons}\; \DefProp,\cons\DefProp\land\DMNeg\DefProp$.%
    }%
    \label{fig:derivations_analytic}
\end{figure}

\begin{example}
The first tree in Figure~{\rm\ref{fig:derivations_analytic}} proves that $\varnothing\;\sequent_{\mathsf{R}\cup\mathsf{R}^\cons}\;\DefProp,\DMNeg\DefProp,\DMNeg\cons\DefProp$ in any $\FmSetAnalytic^\cons$-analytic calculus $\mathsf{R}\cup\mathsf{R}^\cons$ obtained from Theorem {\rm\ref{the:expcalculus}}, and may be
easily built by the algorithm described above.
If we consider the calculus $\mathsf{R}_\FDELogic$ given
in Example {\rm\ref{ex:fdeax}},
the second tree in the same figure shows
an output of the described algorithm
when we search for a countermodel
witnessing $\varnothing\;\notsequent_{\mathsf{R}_{\FDELogic}^{} \cup \mathsf{R}_\FDELogic^\cons}\; \DefProp,\DMNeg\cons\DefProp$.
In this tree, the leftmost branch 
is a non-$\FmSetB$-closed branch
for which no rule instance
based only on subformulas of
$\FmSetA \cup \FmSetB$ and not used yet in the same branch is available.
This implies that,
for $\FmSetD = \{\cons\cons\DefProp,\cons\DefProp,\cons\DMNeg\DefProp,\DMNeg\DefProp\}$, which are the formulas in the leaf of this non-$\FmSetB$-closed branch, we have
$\FmSetD \notsequent_{\mathsf{R}_{\FDELogic}^{} \cup \mathsf{R}_{\FDELogic}^{\cons}}\;\Upsilon^{\FmSetAnalytic^\cons} {\setminus} \FmSetD$.
As the semantical counterpart of this
calculus is the matrix $\langle \PPAlg_6, \Upset{\both} \rangle$, a valuation $v$
such that $v(\FmSetD) \subseteq \Upset{\both}$
necessarily sets $v(\DefProp) = \aefvalue$, since
$\{\DMNeg\DefProp, \cons\DefProp\} \subseteq \FmSetD$.
A similar situation occurs in the third tree,
which constitutes evidence for $\varnothing\;\notsequent_{\mathsf{R}_\FDELogic^{} \cup \mathsf{R}_\FDELogic^\cons}\; \DefProp,\cons\DefProp\land\DMNeg\DefProp$,
meaning that the pseudo-complement 
given by
$\neg x \SymbDef \cons x \land \DMNeg x$ is
non-implosive and, thus, not a classical negation in $\PPlogic$. 
(This is not surprising, in view of Proposition \ref{prop:no_classical_neg}, but it is worth contrasting this with what happens in many other \textbf{LFI}s~{\rm\cite{jm2005thesis}}, in which the latter definition of~$\neg$ \emph{does} correspond to a classical negation.)
\end{example}

\subsection{\SetFmla{} Hilbert-style calculi for logics of De Morgan algebras with prime filters}

We may extend the
recipe given in
Theorem \ref{the:expcalculus},
which delivers a symmetrical Hilbert-style calculus,
to provide a $\SetFmla{}$ Hilbert-style
calculus
for the class $\mathcal{M}^\cons$
when $\mathcal{M}$ itself
is axiomatized by
a $\SetFmla{}$ Hilbert-style calculus.
Before showing how, we will define a collection of such conventional Hilbert-style inference rules associated to a given collection of symmetrical rules.
In what follows,
when $\FmSetA = \{\FmA_1,\ldots,\FmA_n\} \subseteq \LangSet{\Sigma}$ ($n \geq 1$),
let $\bigvee\FmSetA \SymbDef (\ldots(\FmA_1 \lor \FmA_2) \lor \ldots) \lor \FmA_n$. Also, let $\FmSetA \lor \FmB \SymbDef \{\FmA \lor \FmB \mid \FmA \in \FmSetA\}$.
\vspace{0.5em}
\begin{definition}
Let $\mathsf{R}$ be
a symmetrical calculus.
Define the set $\mathsf{R}^{\lor} \SymbDef
\left\{\frac{p}{p \lor q}, \frac{p \lor q}{q \lor p}, \frac{p \lor (q \lor r)}{(p \lor q) \lor r}\right\} \cup
\left\{\mathsf{r}^\lor \mid \mathsf{r} \in \mathsf{R}\right\}$ where $\mathsf{r}^\lor$
is $\frac{\varnothing}{\FmA}$
    if $\mathsf{r} = \frac{\varnothing}{\FmA}$, $\frac{\FmSetA \lor p_0}{(\bigvee \FmSetB) \lor p_0}$
    if $\mathsf{r} = \frac{\FmSetA}{\FmSetB}$,
    and $\frac{\FmSetA \lor p_0}{p_0}$
    if $\mathsf{r} = \frac{\FmSetA}{\varnothing}$,
    where $p_0$ is a
    propositional variable
    not occurring in the
    rules that belong to $\mathsf{R}$.
\end{definition}

The following result 
states that, when
$\mathsf{R}$
is the calculus given
by Theorem \ref{the:expcalculus},
the calculus $\mathsf{R}^\lor$
is the $\SetFmla{}$ Hilbert-style
calculus we are
looking for.
\vspace{.5em}

\begin{theorem}\label{thm:4.3}
\label{the:set-fmla-calculi}
Let $\mathcal{M}$ be a class of $\DMSig$-matrices whose designated sets are prime filters, and let $\mathsf{R}$ be a $\SetFmla{}$ Hilbert-style calculus. If\, $\vdash_\mathsf{R} \;=\; \vdash_\mathcal{M} \;=\; \vdash_{\widehat{\mathcal{M}^\cons}}$, then $\vdash_{(\mathsf{R} \cup \mathsf{R}_\cons)^\lor} \;=\; \vdash_{\mathcal{M}^\cons}$.
\end{theorem}
\vspace{-0.8em}
\begin{proof}If\, $\vdash_\mathsf{R}  \;=\; \vdash_{\widehat{\mathcal{M}}^\cons}$ then $\sequentx{\mathsf{R}} \simeq \sequentx{\widehat{\mathcal{M}^\cons}}$, so,
by Theorem \ref{the:expcalculus}, we have that $\sequentx{\mathcal{M}^\cons} = \sequentx{\mathsf{R} \cup \mathsf{R}_\cons}$, and
thus $\vdash_{\mathcal{M}^\cons} \;=\; \vdash_{\mathsf{R} \cup \mathsf{R}_\cons}$.
Given that $\mathcal{M}$ is a class of $\DMSig$-matrices whose designated sets are prime filters
and $\vdash_\mathsf{R} \;=\; \vdash_\mathcal{M}$, we have
$p\vdash_\mathsf{R} p\lor q$, $q\vdash_\mathsf{R} p\lor q$ and $p\lor q\vdash_\mathsf{R} p, q$.
Since $\widehat{\mathcal{M}^\cons}$ preserves the latter inferences, then $p\vdash_{(\mathsf{R} \cup \mathsf{R}_\cons)} p\lor q$,
and $q\vdash_{(\mathsf{R} \cup \mathsf{R}_\cons)} p\lor q$, and $p\lor q\vdash_{(\mathsf{R} \cup \mathsf{R}_\cons)} p, q$.
The latter statements 
guarantee that $\vdash_{\mathsf{R}}$ has a disjunction, so
by \cite[Theorem 5.37]{shoesmith_smiley:1978} we have that $\vdash_{\mathsf{R} \cup \mathsf{R}_\cons} \;=\; \vdash_{(\mathsf{R} \cup \mathsf{R}_\cons)^\lor}$.
Therefore, $\vdash_{\mathcal{M}^\cons} \;=\; \vdash_{(\mathsf{R} \cup \mathsf{R}_\cons)^\lor}$.
\end{proof}

\begin{example}\label{ex:4.4}
Consider a $\SetFmla{}$ Hilbert calculus 
that axiomatizes $\vdash_\mathcal{B}$.
Since $\mathcal{B} \;=\; \vdash_{\langle \FDEAlg, \Upset{\both} \rangle} \;=\; \vdash_{\widehat{\langle \FDEAlg, \Upset{\both} \rangle^\cons}}$~(cf. {\rm \cite{marcelino2021}}), we obtain a conventional Hilbert-style axiomatization for $\PPlogic \;=\; \vdash_{\langle \FDEAlg, \Upset{\both} \rangle^\cons}\;=\; \vdash_{\langle\mathbf{PP}_6, \uparrow\both\rangle}$ by adding to that calculus the $\mathsf{R}_\cons^\lor$ rules. 
We illustrate, below, with some of the resulting rules:
\begin{align*}
\inferx[\mathsf{r}_6^\lor]
{\cons p \lor r}
{(p \lor \DMNeg p) \lor r}
&&
\inferx[\mathsf{r}_7^\lor]
{\cons p \lor r, p \lor r, \DMNeg p \lor r}
{r}
&&
\inferx[\mathsf{r}_8^\lor]
{\cons p \lor r}
{(\cons(p \land q) \lor p) \lor r}
&&
\inferx[\mathsf{r}_{16}^\lor]
{\cons p \lor r, p \lor r}
{\cons(p \lor q) \lor r}
\end{align*}
\end{example}

In what follows we consider a few extensions of $\PPlogic$,
illustrating how our methods may be used to axiomatize them.
The following result, which is an immediate consequence of Theorem \ref{thm:4.3},
shows that Example \ref{ex:4.4} smoothly generalises to all super-Belnap logics.

\begin{proposition}\label{prop:4.5}
Let $\mathcal{M}$ be a class of models of $\mathcal{B}$ whose designated sets are prime filters. If $\vdash_\mathcal{M}$
is ax\-iom\-atized relative to $\FDELogic$ by a set~$\mathsf{R}$ of \SetFmla{} rules, 
then $\vdash_{\DMPPExtend{\mathcal{M}}}$ is also axiomatized by $\mathsf{R}$ relative to~$\DMPPExtend{\mathcal{B}}$.
\end{proposition}

Let $\mathcal{M}_1$ and $\mathcal{M}_2$ be two classes of models of $\mathcal{B}$
such that $\vdash_{\mathcal{M}_1} \;=\; \vdash_{\mathcal{M}_2}$.
Then $\vdash_{\mathcal{M}_1}$ and $\vdash_{\mathcal{M}_2}$ are axiomatized
by the same set $\mathsf{R}$ of singleton-succedent rules.
Hence, $\vdash_{\DMPPExtend{\mathcal{M}_1}} \;=\; \vdash_{\DMPPExtend{\mathcal{M}_2}}$
is axiomatized by the set $\mathsf{R}_\cons^\lor$ defined above.
This entails, in particular, that,
if a super-Belnap logic $\SuperBelnapCon$ is finitary, then $\SuperBelnapCon^\cons$
(described in Lemma \ref{coro:conservsuper}) is also finitary. 
Since the lattice of super-Belnap logics contains continuum-many finitary logics \cite[Corollary 8.17]{prenosil2021}, we obtain the following sharpening of Corollary \ref{coro:embdsupbelnap}:

\begin{proposition}
There are continuum-many finitary extensions of $\PPlogic$.
\end{proposition}

The super-Belnap logics (see \cite{rivieccio2017} for further details) considered below for the sake of illustration are the
Asenjo-Priest Logic of Paradox $\mathcal{LP}$, the two logics $\mathcal{K}_\leq$ and $\mathcal{K}_1$ named after S.\ C.\ Kleene, and Classical
Logic $\mathcal{CL}$.
In the remaining results of this section, we use the notation $\mathcal{L} + (R)$ to refer to the \SetFmla{} Hilbert-style calculus
resulting from adding rule $(R)$ to
a Hilbert-style system for $\mathcal{L}$.
{In addition, we will write $\Logic\mathcal{M}$
for $\vdash_\mathcal{M}$.}
The next result establishes that each of these logics can be axiomatized,
relative to $\mathcal{B}$, by a combination of the rules given below.
In the sequel, we show, in a similar way, 
how some logics characterized
by matrices over PP-algebras can be
axiomatized relatively to $\PPlogic$.
\begin{align*}
\inferx[(\mathrm{DS})]
{p \land (\dmneg p \lor q)}
{q}
&&
\inferx[({K}_1)]
{(p \land \dmneg p) \lor q}
{q}
&&
\inferx[({K}_\leq)]
{(p \land \dmneg p) \lor r}
{q \lor \dmneg q \lor r}
&&
\inferx[\mathrm{(EM)}]
{}
{p \lor \dmneg p}
\end{align*}

\begin{proposition}{\rm(\cite[Theorem 3.4]{rivieccio2017})}\label{prop:4.6}
\begin{enumerate}\itemsep0pt
    \item[(i)] $\mathcal{LP} = \Logic\langle\mathbf{K}_3, {\uparrow}\neither\rangle = \mathcal{B} + (EM)$
    \item[(ii)] $\mathcal{K}_1 = \Logic\langle\mathbf{K}_3, \{\tvalue\}\rangle = \mathcal{B} + (K_1)$
    \item[(iii)] $\mathcal{K}_\leq = \Logic\{\langle\mathbf{K}_3, {\Upset{\neither}}\rangle, \langle\mathbf{K}_3, \{\tvalue\}\rangle\} = \mathcal{B} + (K_\leq)$
    \item[(iv)] $\mathcal{CL} = \Logic\langle\mathbf{B}_2, \{\tvalue\}\rangle = \mathcal{B} + (DS) + (EM)$
\end{enumerate}
\end{proposition}

\begin{theorem} For logics above $\PPlogic$ we have the following relative axiomatizations:
\begin{enumerate}\itemsep0pt
    \item[(i)] $\Logic\langle\mathbf{PP}_5, {\uparrow}\neither\rangle = \mathcal{LP}^\cons = \PPlogic + (EM)$
    \item[(ii)] $\Logic\langle\mathbf{PP}_5, {\uparrow}\tvalue\rangle = \mathcal{K}_1^\cons = \PPlogic + (\mathcal{K}_1)$
    \item[(iii)] $\Logic\{\langle\mathbf{PP}_5, {\uparrow}\neither\rangle, \langle\mathbf{PP}_5, {\uparrow}\tvalue\rangle\} = \mathcal{K}_\leq^\cons = \PPlogic + (K_\leq)$
    \item[(iv)] $\Logic\langle\mathbf{PP}_4, {\uparrow}\tvalue\rangle = \mathcal{CL}^\cons = \PPlogic + (DS) + (EM)$
\end{enumerate}
\end{theorem}
\begin{proof}
\sloppy
This follows directly from Proposition \ref{prop:4.5} and Proposition \ref{prop:4.6}, taking into account that 
$\langle\mathbf{B}_2, \{\tvalue\}\rangle^\cons = \langle\mathbf{PP}_4, {\uparrow}\tvalue\rangle$, and
for $x \in \{\tvalue, \neither\}$, we have $\langle\mathbf{K}_3, {\uparrow}x\rangle^\cons = \langle\mathbf{PP}_5, {\uparrow}x\rangle$.
\end{proof}%
\section{Final remarks}
\label{sec:finalremarks}


We have seen how to endow with a perfection connective logics characterized by matrices having a De Morgan algebraic reduct, offering two possible directions: either by appropriately expanding the corresponding matrices or by adding new rules of inference to an existing Hilbert-style axiomatization. 
In particular, by so enriching Dunn-Belnap's 4-valued logic we obtained the 6-valued order-preserving logic $\PPlogic$, associated to the variety of expanded algebras, which we called `perfect paradefinite algebras' 
and proved to be term-equivalent with the variety of involutive Stone algebras. 
It is worth mentioning that the one-one correspondence between both varieties can be used to introduce back-and-forth functors
that establish a categorical equivalence
between the corresponding algebraic categories.

By providing a Derivability Adjustment Theorem for $\PPlogic$ and its extensions, we have also shown that Boolean reasoning is fully recovered using De Morgan negation and the perfection operator. 
Notice, indeed, that adding the equation $\cons x \approx \top$ to a perfect paradefinite algebra, intuitively stating 
that every element is
Boolean,
what results 
in an
algebra that is (term-equivalent to) a Boolean algebra.

The equational basis of the variety of PP-algebras studied here was conceived having in mind the expected term-equivalence with the variety of IS-algebras. 
A natural path for future work is to drop this constraint and study De Morgan algebras enriched with perfection operators satisfying weaker equations,
and the corresponding logics.
Within such more general algebraic structures and the logics based thereon, and taking also into account the proposed comparison between  $\PPlogic$ and other \textbf{C}-systems and \textbf{D}-systems in the literature, one may for instance  consider two distinct negations (not necessarily respecting all De Morgan laws) instead of a single one that is at once paraconsistent and paracomplete, and possibly also two separate `recovery connectives', as, e.g., in~\cite{dodo2014, lahav2017}.

Yet another direction for future investigation would be to  enrich~$\PPlogic$ with an implication connective;
similar paths have recently been explored in~\cite{esteva2021}, but considering involutive distributive residuated lattices instead of De Morgan algebras.
A promising starting point for this research could be provided by the following observation. 
The algebra $\PPSix$, as a finite distributive lattice, has
an implicitly definable (and unique) intuitionistic implication given by  the relative pseudo-complement. By adding this operation to the propositional language and considering the corresponding logical matrix, one  thus obtains a conservative extension of $\PPlogic$ by the intuitionistic implication. We expect this logic to be algebraizable, but not necessarily self-extensional. On the other hand, a (weaker) self-extensional conservative extension of $\PPlogic$ may be obtained by considering the logic determined by the family of all matrices of type $\langle \PPSix, D \rangle$
based on $\PPSix$ (endowed with the relative pseudo-complement operation), with $D$ a lattice filter. 
Logics obtained in this way will have
as  algebraic counterparts (subclasses of) 
algebras that 
 carry both
a De Morgan negation and an intuitionistic implication: these structures have been studied in the literature under the names of 
\emph{symmetric Heyting algebras} (A.~Monteiro~\cite{Monteiro1980}) 
and \emph{De Morgan-Heyting algebras}
(H.P.~Sankappanavar~\cite{Sankappanavar1987}). From a technical point of view, an advantage of the proposed  approach is thus one may hope
to be able to
import results
from the well-developed theory of the above-mentined classes of algebras.
\bibliographystyle{eptcs}
\bibliography{bib}
\end{document}